\documentclass{amsart}
\usepackage{amsmath,amssymb,amscd,amsthm,latexsym}

\usepackage{mathrsfs}

\usepackage{stmaryrd}

\usepackage[all]{xy}

\usepackage{graphicx}

\usepackage[usenames, dvipsnames]{color}

\definecolor{NTNUblue}{RGB}{0,80,158}
\definecolor{NTNUbluesupport}{RGB}{62,98,138}
\definecolor{NTNUorange}{RGB}{239,129,20}

\usepackage{hyperref}

\usepackage[normalem]{ulem}
\usepackage{mathtools}

\newcommand{\bi}{\begin{itemize}}
\newcommand{\ei}{\end{itemize}}
\newcommand{\bn}{\begin{enumerate}}
\newcommand{\en}{\end{enumerate}}

\newcommand{\bq}{\begin{equation}}
\newcommand{\eq}{\end{equation}}

\newcommand{\C}{{\mathbb{C}}}

\newcommand{\ED}{E_{\Dh}}
\newcommand{\HD}{H_{\Dh}}
\newcommand{\MUD}{MU_{\Dh}}

\newcommand{\Q}{{\mathbb{Q}}}

\newcommand{\Pro}{{\mathbb{P}}}
\newcommand{\CP}{{\mathbb{CP}}}

\newcommand{\R}{{\mathbb{R}}}
\newcommand{\Z}{{\mathbb{Z}}}
\newcommand{\cl}{\mathrm{cl}}

\newcommand{\codim}{\mathrm{codim}\,}

\newcommand{\diff}{\mathrm{diff}}

\newcommand{\ev}{\mathrm{ev}}

\newcommand{\geo}{\mathrm{geo}}

\newcommand{\Hdg}{\mathrm{Hdg}}
\newcommand{\HdgE}{\mathrm{Hdg}_{E}}
\newcommand{\HdgMU}{\mathrm{Hdg}_{MU}}

\newcommand{\hol}{\mathrm{hol}}

\newcommand{\id}{\mathrm{id}}

\newcommand{\inc}{\mathrm{inc}}

\newcommand{\KK}{\Kh}

\newcommand{\pr}{\mathrm{pr}}

\newcommand{\pt}{\mathrm{pt}}

\newcommand{\topp}{\mathrm{top}}

\newcommand{\coker}{\mathrm{coker}\,}

\newcommand{\CS}{\mathrm{CS}}

\newcommand{\KCS}{\CS_K}

\newcommand{\Hom}{\mathrm{Hom}}
\newcommand{\Imm}{\mathrm{Im}\,}

\newcommand{\Manc}{\mathbf{Man}_{\C}}

\newcommand{\ManF}{\mathbf{Man}_F}

\newcommand{\Ah}{{\mathcal A}}

\newcommand{\Dh}{{\mathcal D}}
\newcommand{\Eh}{{\mathcal E}}

\newcommand{\xto}{\xrightarrow}

\newcommand{\Ds}{\mathscr{D}}

\newcommand{\Kh}{\mathcal{K}}

\newcommand{\Mh}{{\mathcal M}}
\newcommand{\Nh}{{\mathcal N}}
\newcommand{\Oh}{{\mathcal O}}

\newcommand{\Vh}{{\mathcal V}}

\newcommand{\Zh}{\mathcal{Z}}

\newcommand{\oC}{\overline{C}}

\newcommand{\oU}{\overline{U}}

\newcommand{\into}{\hookrightarrow}

\newcommand{\hfcycle}{\gamma}
\newcommand{\geocycle}{\widetilde f}
\newcommand{\geocycles}{\widetilde{ZMU}}
\newcommand{\geocob}{\widetilde b}

\newcommand{\trivbc}{\underline{\mathbb C}}

\newcommand{\MUDor}{\mathfrak{o}}
\newcommand{\MUDors}{K_{\MUD}}
\newcommand{\MUDorrep}{\epsilon}

\DeclareMathOperator*{\cone}{cone}

\newtheorem{theorem}{Theorem}[section]
\newtheorem{lemma}[theorem]{Lemma}
\newtheorem{prop}[theorem]{Proposition}

\theoremstyle{definition}
\newtheorem{defn}[theorem]{Definition}

\newtheorem{remark}[theorem]{Remark}

\begin{document}

\author{Knut Bjarte Haus}
\address{Department of Mathematical Sciences, NTNU, NO-7491 Trondheim, Norway}
\email{knut.b.haus@gmail.com}

\author{Gereon Quick}
\thanks{The second-named author was partially supported by the RCN Project No.\,313472 {\it Equations in Motivic Homotopy}.} 
\address{Department of Mathematical Sciences, NTNU, NO-7491 Trondheim, Norway}
\email{gereon.quick@ntnu.no}

\title[Geometric pushforward and secondary cobordism invariants]{Geometric pushforward in Hodge filtered complex cobordism and secondary invariants}

\date{}

\begin{abstract}
We construct a functorial pushforward homomorphism in geometric Hodge filtered complex cobordism along proper holomorphic maps between arbitrary complex manifolds. 
This significantly improves previous results on such transfer maps and is a much stronger result than the ones known for differential cobordism of smooth manifolds. 
This enables us to define and provide a concrete geometric description of Hodge filtered fundamental classes for all proper holomorphic maps. Moreover, we give a geometric description of a cobordism analog of the Abel--Jacobi invariant for nullbordant maps which is mapped to the classical invariant under the Hodge filtered Thom morphism.  
For the latter we provide a new construction in terms of geometric cycles. 
\end{abstract}


\maketitle

\tableofcontents

\section{Introduction}
\label{sec:intro}

The study of the analytic submanifolds of a given compact K\"ahler manifold is a central theme in complex geometry. 
Fundamental classes provide important invariants for this study. 
For a classical example, 
let $X$ be a compact K\"ahler manifold and $Z \subset X$ a submanifold of codimension $p$.  
The Poincar\'e dual of the pushforward of the fundamental class of $Z$ along the inclusion defines a cohomology class $[Z]$ in $H^{2p}(X;\Z)$. 
In fact, $[Z]$ lies in the subgroup $\Hdg^{2p}(X)=H^{2p}(X;\mathbb Z)\cap H^{p,p}(X;\C)$ of integral classes of Hodge type $(p,p)$. 
This induces a homomorphism from the free abelian group $\Zh^p(X)$ generated by submanifolds of codimension $p$ of $X$ to $\Hdg^{2p}(X)$. 
This map lifts to a homomorphism to Deligne cohomology $H_{\Dh}^{2p}(X;\Z(p))$. 
The latter group fits in the short exact sequence 
\begin{align}\label{fundamentalSESDeligne}
0\to J^{2p-1}(X) \to H^{2p}_\Dh(X;\Z(p)) \to \Hdg^{2p}(X)\to 0
\end{align}
where $J^{2p-1}(X)$ denotes Griffiths' intermediate Jacobian (see for example \cite[\S 12]{voisin1}).  
On the subgroup $\Zh_{\hom}^p(X)$ of submanifolds whose fundamental class is homologically trivial sequence \eqref{fundamentalSESDeligne} induces the Abel--Jacobi map $\Zh_{\hom}^p(X) \to J^{2p-1}(X)$. 
As described in \cite[\S 12.1]{voisin1} this map has a concrete geometric description via evaluating integrals over singular cycles in $X$, and one may consider it as a secondary cohomology invariant.  
In \cite{karoubi43,karoubi45} Karoubi constructed an analog of Deligne cohomology for complex $K$-theory over complex manifolds in which secondary invariants for vector bundles can be defined (see also \cite{Esnault} for a study of induced secondary invariants). 
In \cite{hfc} the authors show that there is a bigraded analog of Deligne cohomology $\ED$ for every rationally even cohomology theory $E$. 
If $X$ is a compact K\"ahler manifold, there is a short exact sequence
\begin{align*}
0\to J_{E}^{2p-1}(X) \to E_\Dh^{2p}(p)(X)\to \HdgE^{2p}(X)\to 0
\end{align*}
generalizing sequence \eqref{fundamentalSESDeligne}. 
Let $X$ be a smooth projective complex algebraic variety and $\widetilde{\Mh}^p(X)$ be the free abelian group generated by isomorphism classes $[f]$ of projective smooth morphisms $f \colon Y \to X$ of codimension $p$ between complex algebraic varieties. 
Based on the work of Levine and Morel \cite{lm} on algebraic cobordism, it is shown in \cite{hfc} that for $E=MU$ there is a natural homomorphism $\widehat{\varphi} \colon \widetilde{\Mh}^p(X) \to \MUD^{2p}(p)(X)$ where $X$ also denotes the underlying complex manifold of complex points of $X$. 
On the subgroup $\widetilde{\Mh}^p(X)_{\topp}$ of topologically cobordant maps this induces an Abel--Jacobi type homomorphism $AJ \colon \widetilde{\Mh}^p(X)_{\topp} \to J_{MU}^{2p-1}(X)$. 
This homomorphism has been studied in more detail in \cite{aj}. 
However, both $\widehat{\varphi}$ and $AJ$ are only defined for complex algebraic varieties and are not induced by a geometric procedure as their classical analogs, but by a rather abstract machinery. \\

In \cite{ghfc} the authors define for every complex manifold $X$ and integers $n$ and $p$, geometric Hodge filtered complex cobordism groups $MU^n(p)(X)$ recalled below.
The main result of \cite{ghfc} is that there is a natural isomorphism of Hodge filtered cohomology groups
\begin{align}\label{eq:ghfc_mainthmiso}
\MUD^n(p)(X) \cong MU^n(p)(X). 
\end{align}
%
The aim of the present paper is to construct pushforward homomorphisms along proper holomorphic maps for geometric Hodge filtered cobordism. 
This will allow us to give a concrete description of the Hodge filtered fundamental classes of holomorphic maps $f\colon Y \to X$ for any complex manifold $X$ 
and of the Abel--Jacobi invariant $AJ$ for topologically trivial cobordism cycles on compact K\"ahler manifolds. 
We note that the results of the present paper are independent of the comparison isomorphism \eqref{eq:ghfc_mainthmiso} of \cite{ghfc}. 
The only results from \cite{ghfc} that we assume here are the verification of some natural properties of the groups $MU^n(p)(X)$. \\

We will now briefly describe the construction of the groups $MU^n(p)(X)$ of \cite{ghfc} which we also recall in more detail in section \ref{Section:GeometricHFCBordism} and will then describe our main results in more detail. 
Consider the genus $\phi \colon MU_*\to \Vh_*:=MU_*\otimes_{\Z} \C$ given by multiplication by $(2\pi i)^{n}$ in degree $2n$. By Thom's theorem, $MU_n$ is the bordism group of $n$-dimensional almost complex manifolds $Z$. 
Hirzebruch showed that 
any genus $\phi\colon MU_*\to \Vh_*$ is of the form 
\[
\phi(Z)=\int_Z(K^\phi(TZ))^{-1}
\]
for a multiplicative sequence $K^\phi$, where $TZ$ denotes the tangent bundle of $Z$. 
This yields a $\Vh_*$-valued characteristic class of complex vector bundles. 
For $p\in \Z$, we consider the characteristic class $K^p=(2\pi i)^p\cdot K^\phi$. 
If $\nabla$ is a connection on a complex vector bundle $E$, Chern--Weil theory gives a form $K^p(\nabla)$ representing $K^p(E)$. 
Given a form $\omega$ on $Z$ and a proper oriented map $f\colon Z\to X$, we consider the pushforward current $f_*\omega$, which acts on compactly supported forms on $X$ by $\sigma\mapsto \int_Z\omega\wedge f^*\sigma$. 
We define Hodge filtered cobordism cycles as triples $(f,\nabla,h)$ where $f$ is a proper  complex-oriented map $f\colon Z\to X$, $\nabla$ is a connection on the complex stable normal bundle of $f$ and $h$ is a current on $X$ such that 
\[
f_*K^p(\nabla)-dh ~ \text{is a smooth form in} ~ F^p\Ah^n(X;\Vh_*).
\] 
After defining a suitable Hodge filtered bordism relation we then obtain the group $MU^n(p)(X)$ of Hodge filtered cobordism classes.  
The main new technical contribution of the present paper is the construction of pushforwards for geometric Hodge filtered complex cobordism. 

\vspace{.1cm} 

\begin{theorem}\label{thm:pushforwardmapsexistintro}
Let $g \colon X\to Y$ be a proper holomorphic map between complex manifolds of complex codimension $d=\dim_\C Y-\dim_\C X$. 
Then there is a pushforward homomorphism of $MU^*(*)(Y)$-modules  
\[
g_*\colon MU^n(p)(X)\to MU^{n+2d}(p+d)(Y)
\]
which is functorial for proper holomorphic maps and compatible with pullbacks. 
\end{theorem} 

\vspace{.1cm}

In \cite[Section 7]{hfc}, the authors show that there is an $\MUD$-pushforward along 
projective morphisms between smooth projective complex varieties. 
In fact, they show that there are pushforward maps for a logarithmically refined version of $\MUD$ for quasi-projective smooth complex varieties, as a rather formal consequence of the projective bundle formula. 
This theory coincides with $\MUD$ for projective smooth complex varieties. 
Hence, assuming the comparison isomorphism \eqref{eq:ghfc_mainthmiso} of \cite{ghfc}, Theorem \ref{thm:pushforwardmapsexistintro} extends the existence of pushforwards to a significantly larger class of maps than the one in \cite{hfc}. 
Since pushforwards in cohomology $g_* \colon MU^n(X)\to MU^{n+2d}(Y)$ only exist for proper and complex oriented continuous maps of (real) codimension $d$, the class of proper holomorphic maps is the largest possible subclass of holomorphic maps for which a pushforward with good properties may exist. \\


The construction of $g_*$ in Theorem \ref{thm:pushforwardmapsexistintro} is similar to the one of pushforwards for differential cobordism for smooth manifolds in \cite{Bunke2009}. 
However, the pushforward in differential cobordism exists only for proper submersions with a choice of a  smooth $MU$-orientation. 
We will now explain why the  pushforward for Hodge filtered cobordism exists for all proper holomorphic maps. 
In section \ref{sec:MUDorientations} we first define 
the group of 
a Hodge filtered $MU$-orientation as a Grothendieck group of triples $(E,\nabla,\sigma)$ where $E$ is a complex vector bundle with connection $\nabla$ and $\sigma$ is a form on $X$ such that 
$K(\nabla)-d\sigma \in F^0\Ah^0(X;\Vh_*)$. 
The relations involve a Chern--Simons transgression form associated to the multiplicative sequence $K$.  
A Hodge filtered $MU$-oriented map is then a holomorphic map with a lift of the stable normal bundle to the group of Hodge filtered $MU$-orientations. 
Then we show in section \ref{sec:pushforward_along_MUDOrientations} that there is a pushforward along every proper Hodge filtered $MU$-oriented map. 
Finally, we show in section \ref{sec:canonical_MUD_orientation} that there is a canonical choice of a Hodge filtered $MU$-orientation for every proper holomorphic map. 
The key idea is a variation of a result of Karoubi's \cite[Theorem 6.7]{karoubi43} which establishes a mapping of virtual holomorphic vector bundles to the group of Hodge filtered $MU$-orientations by picking a Bott connection.\footnote{The notion of Bott connection in \cite[\S 6]{karoubi43} is a generalization of Bott connections in  foliation theory. 
For complex manifolds, Bott connections are connections compatible with the holomorphic structure.}  
Applying this result to the virtual normal bundle of a holomorphic map defines a canonical orientation which we call the \emph{Bott orientation}. \\

Another crucial point for the construction of pushforwards is that there is a currential version of Hodge filtered cobordism which we introduce in section \ref{sec:currentialHFC}.  
A key difference to differential cohomology theories on smooth manifolds, such as differential cobordism or differential $K$-theory, is that, for Hodge filtered cobordism, the currential description and the one using forms are canonically isomorphic. 
This is not the case for differential theories as explained for differential $K$-theory in \cite{FreedLott}, where particularly the exact sequences \cite[(2.20)]{FreedLott} and \cite[(2.29)]{FreedLott} make it clear that the smooth and currential differential $K$-theory groups are different in general.
The main reason is that the space of closed currents $\Ds^n(X)_{\cl}$ is strictly larger than the space of closed forms $\Ah^n(X)_{\cl}$. 
In the Hodge filtered context, however, 
we use $H^n(X;F^p\Ah^*)$ and $H^n(X;F^p\Ds^*)$, and 
$H^n(X;\Ah^*/F^p)$ and $H^n(X;\Ds^*/F^p)$  instead of $\Ah^n(X)/\Imm(d)$ and $\Ds^n(X)/\Imm(d)$.
Since the Dolbeault--Grothendieck lemma holds both for currents and forms, in both cases the canonical map from the first to the second group is an isomorphism (see Lemma \ref{lem:DolbeaultGrothendieck} and Theorem \ref{equivGeoDesc}). \\

We will now describe the remaining content and results of the paper. 
In section \ref{sec:pushforward_along_MUDOrientations} we show that the pushforward is functorial, compatible with pullbacks and satisfies a projection formula. 
In section \ref{sec:fund_and_secondary_classes} we introduce the Hodge filtered fundamental class $[f]=f_*(1) \in MU^{2p}(p)(X)$ associated to a holomorphic map $f \colon Y \to X$ of codimension $p$ as the pushforward of the unit element along $f$. 
If $X$ is a compact K\"ahler manifold and $f$ is a nullbordant proper holomorphic map, then $[f]$ has image in the subgroup $J_{MU}^{2p-1}(X)$ which has the structure of a complex torus.  
In this case we also write $AJ(f):=[f]$ for the class of $f$ in  $J_{MU}^{2p-1}(X)$. 
Since the pushforward $f_*$ has a geometric construction, we are able to give a geometric description of the secondary invariant $AJ(f)$ in section \ref{sec:fund_and_secondary_classes}. 
The main ingredient in the formulas are  certain Chern--Simons transgression forms mediating between an arbitrary connection on the normal bundle $N_f$ and Bott connections on the corresponding tangent bundles. 
In section \ref{section:hfThomMorphism} we present a cycle model for Deligne cohomology inspired by but slightly simpler than the one of Gillet and Soulé in \cite{GilletSoule}. 
The main difference is that we use currents of integration instead of integral currents in the sense of geometric measure theory (see also \cite{HarveyLawson}). 
The new construction may be of independent interest and useful for other applications. 
This enables us to give a cycle description of the Hodge filtered Thom morphism $\MUD^n(p)(X)\to H_\Dh^n(X;\Z(p))$ for every complex manifold $X$ and integers $n$ and $p$. 
In section \ref{section:hfThomMorphism_Kahler} we report on our knowledge of the current status of examples and phenomena related to the kernel and image of the Hodge filtered Thom morphism for compact K\"ahler manifolds. \\

Parts of the present paper grew out of the work on  the first-named author's doctoral thesis. 
Both authors would like to thank the Department of Mathematical Sciences at NTNU for the continuous support during the work on the thesis and this paper. 
We thank the anonymous referee for many comments and valuable suggestions which helped improving the paper. 


\section{Currential geometric Hodge filtered cobordism} \label{Section:GeometricHFCBordism}

First we briefly recall some facts about currents and the construction of geometric Hodge filtered complex cobordism groups from \cite{ghfc}. 
Then we introduce a currential version of Hodge filtered cobordism.

\subsection{Currents}

Let $X$  be a smooth manifold and let $\Lambda_X$ denote the orientation bundle of $X$. 
Let $\Ah_c^*(X;\Lambda_X)$ be the space of compactly supported smooth forms on $X$ with values in $\Lambda_X$.
Let $\Ds^*(X)$ denote the space of currents on $X$, defined as the topological dual of $\Ah_c^*(X;\Lambda_X)$. 
Given a form $\omega\in\Ah^*(X)$ and a current $T\in \Ds^*(X)$, their product acts by
\[
T\wedge \omega (\sigma)= T(\omega\wedge \sigma).
\] 
There is an injective map $\Ah^*(X) \into \Ds^*(X)$ given by
\[
\omega \mapsto T_\omega =\left(\sigma\mapsto \int_X\omega\wedge \sigma , ~~ \sigma\in \Ah^*_c(X;\Lambda_X)\right).
\]
We equip $\Ds^*$ with a grading so that this injection preserves degrees. 
That is, $\Ds^k(X)$ consists of the currents which vanish on a homogeneous $\Lambda_X$-valued form $\sigma$, 
unless possibly if $\deg \sigma = \dim_\R X-k$. 
We will not always distinguish $\omega$ from $T_\omega$ in our notation.

If $X$ is a manifold without boundary, Stokes' theorem implies for $\omega\in \Ah^k(X)$: 
\[
T_{d\omega}(\sigma)=(-1)^{k+1}T_\omega(d\sigma).
\]
Hence the exterior differential can be extended to a map $d\colon \Ds^k(X)\to \Ds^{k+1}(X)$ by 
\[
dT(\sigma)=(-1)^{k+1}T(d\sigma). 
\]
For a vector space $V$ we set $\Ds^*(X;V)=\Ds^*(X)\otimes V$, and for an evenly graded complex vector space $\Vh_*=\oplus_j \Vh_{2j}$ we set 
\[
\Ds^n(X;\Vh_*)=\bigoplus_j \Ds^{n+2j}(X;\Vh_{2j}).
\]
An orientation of a map $f\colon Z \to X$ is equivalent to an isomorphism $\Lambda_Z \cong f^*\Lambda_X$. 
If $f$ is proper and oriented, we therefore get a map 
\[
f^*\colon \Ah_c^*(X;\Lambda_X)\to \Ah_c^*(Z;\Lambda_Z)
\] 
which induces a homomorphism 
\[
f_*\colon\Ds^*(Z)\to \Ds^{*+d}(X)
\]
where $d=\codim f=\dim X-\dim Z$.
We also denote by $f_*$ the homomorphism $\Ds^*(Z;\Vh_*)\to \Ds^*(X;\Vh_*)$ induced by tensoring $f_*$ with the identity of the various $\Vh_{2j}$. 
We then have the identity 
\[
d \circ f_*=(-1)^d f_*\circ d.
\]

\begin{remark}\label{IntegrationOverTheFiber}
In the case of a submersion $\pi\colon W\to X$ the pushforward $\pi_*$ preserves smoothness. 
We thus obtain the \emph{integration over the fiber} map
\[
\int_{W/X}\colon \Ah^*(W)\to \Ah^{*+d}(X)
\] 
defined by the equation
\[
T_{\int_{W/X}\omega} = \pi_*T_\omega.
\]
\end{remark}

Now we assume that $X$ is a complex manifold. 
Then the space of currents is bigraded as follows. 
We write $\Ds^{p,q}(X)$ for the subgroup of those currents which vanish  on compactly supported $(p',q')$-forms unless $p'+p=\dim_\C X=q'+q$. 
Then $\Ah^{*,*}(X)\to \Ds^{*,*}(X)$, $\omega \mapsto T_\omega$, is a morphism of double complexes. 
The Hodge filtration on currents is defined by
\begin{align*}
F^p\Ds^n(X) &= \bigoplus_{i\geq p} \Ds^{i,n-i}(X). 
\end{align*}
For an evenly graded complex vector space $\Vh_*$ we set 
\begin{align*}
F^p\Ds^n(X;\Vh_*) := \bigoplus_j F^{p+j}\Ds^{n+2j}(X;\Vh_{2j}), ~ 
\frac{\Ds^n}{F^p}(X;\Vh_*) := \bigoplus_j  \frac{\Ds^{n+2j}(X;\Vh_{2j})}{F^{p+j}\Ds^{n+2j}(X;\Vh_{2j})}.
\end{align*}
With similar notation for $F^p\Ah^*(X;\Vh_*)$ and $\frac{\Ah^*}{F^p}(X;\Vh_*)$ we get the following result which will be crucial for the proof of Theorem \ref{equivGeoDesc}:

\begin{lemma}
\label{lem:DolbeaultGrothendieck}
Let $X$ be a complex manifold.
For every $p$, the maps of complexes of sheaves 
$F^p\Ah^*(\Vh_*) \to F^p\Ds^*(\Vh_*)$ and 
$\frac{\Ah^*}{F^p}(\Vh_*) \to \frac{\Ds^*}{F^p}(\Vh_*)$ are quasi-isomorphisms on the site of open subsets of $X$.
\end{lemma}

\begin{proof}
It suffices to prove the assertions for $\Vh_*=\C$. 
Let $\Omega^p$ be the sheaf of holomorphic $p$ forms on $X$. 
The maps of complexes  
\[
\Omega^p \to \Ah^{p,*}\to \Ds^{p,*}
\]
are quasi-isomorphisms by 
the Dolbeault--Grothendieck Lemma as formulated and proven in \cite[pages 382--385]{GriffithsHarris} (see also  \cite[Lemma 3.29 on page 28]{Demailly}), where we consider $\Omega^p$ as a complex concentrated in a single degree.  
%
The sheaves $\Ah^{p,q}$ and $\Ds^{p,q}$, being modules over $\Ah^0$, are \emph{fine}. 
Therefore, it follows from \cite[Lemma 8.5]{voisin1} that the solid inclusions
\begin{align*}
\xymatrix{
& \ar[dl] \Omega^{*\geqslant p} \ar[dr] & \\ 
F^p\Ah^* \ar@{.>}[rr] & & F^p\Ds^*  
}
\end{align*}
are quasi-isomorphisms. 
This implies that the dotted arrow is a quasi-isomorphism and proves the first assertion. 
By the same argument we have de Rham's theorem and the inclusion $\Ah^* \to \Ds^*$ is a quasi-isomorphism. 
Hence the induced map between the cokernels of the maps $\Ah^* \to \Ds^*$ and $F^p\Ah^* \to F^p\Ds^*$ is a quasi-isomorphism as well. 
This proves the second assertion. 
\end{proof}


\subsection{Geometric Hodge filtered cobordism groups}

We briefly recall the construction of the geometric Hodge filtered cobordism groups of \cite{ghfc}. 
For further details we refer to \cite[\S 2]{ghfc}.  
Let $\Manc$ denote the category of complex manifolds with holomorphic maps. 
For $X\in \Manc$ let $f\colon Z \to X$ be a proper complex-oriented map, and let $N_f$ be a complex vector bundle which represents the stable normal bundle of $f$. 
Let $\nabla_f$ be a connection on $N_f$. 
We call the triple $\geocycle=(f,N_f,\nabla_f)$ a \emph{geometric cycle} over $X$. 
We let $\geocycles^n(X)$ 
denote the abelian group generated by isomorphism classes, in the obvious sense, of geometric cycles over $X$ of codimension $n$ with the relations $\geocycle_1+\geocycle_2 = \geocycle_1\sqcup \geocycle_2$.


Let $MU_*$ be the graded ring with $MU_n=MU^{-n}(\pt)$. 
A map of rings $MU_*\to R$ for $R$ an integral domain over $\Q$ is called a complex genus. 
Complex genera may be constructed in the following way. 
For each $i\in \mathbb N$, let $x_i$ be an indeterminate of degree $i$. 
Let $Q\in R[[y]]$ be a formal power series in the variable $y$ of degree $2$. 
Let $\sigma_i$ denote the $i$-th elementary symmetric function in $x_1,x_2,\dots $. 
We may then define a sequence of polynomials $K_i^Q$ satisfying 
\[
K^Q(\sigma_1,\sigma_2,\cdots)=1+K^Q_2(\sigma_1)+K^Q_4(\sigma_1,\sigma_2)+\dots = \prod_{i=1}^\infty Q(x_i)
\]
since the right-hand side is symmetric in the $x_i$.  
Then we get a characteristic class $K^Q$ defined on a complex vector bundle $E \to X$ of dimension $n$ by
\[
K^Q(E):=K^Q(c_1(E),\dots, c_n(E))\in H^*(X;R)
\]
where $c_i(E)$ denotes the $i$-th Chern class of $E$. 
In fact, 
by \cite[\S 1.8]{HBJ}, all genera are of the form
\[
\phi^Q([X])=\int_XK^Q(N_X)
\]
where $N_X$ denotes the complex vector bundle representing the stable normal bundle of $X$ obtained from the complex orientation of $X\to \pt$. 
From now on we set $\Vh_*:=MU_*\otimes_{\Z}\C$. 
We assume that the power series $Q(y)=1+r_1y+r_2y^2+\cdots $ has total degree $0$. 
This is equivalent to assuming $\phi^Q$ to be a degree-preserving genus. 
Then $K^Q(E)$ has total degree $0$.  
By \cite[Lemma 3.26]{Bunke2009}, $\phi^Q$ extends to a morphism of multiplicative cohomology theories 
\[
\phi^Q\colon MU^n(X)\to H^n(X;\Vh_*)
\]
by 
\begin{align*}
\phi^Q([f])=f_*K^Q\left(N_f\right).
\end{align*}
Here $H^n(X;\Vh_*)\cong \bigoplus_j H^{n+2j}(X;\Vh_{2j})$, so that in particular $H^{-2j}(\pt;\Vh_*) \cong  \Vh_{2j}$. 
%
Now we fix the multiplicative natural transformation  
\begin{align*}
\phi\colon MU^*(X)\to H^*(X;\Vh_*)
\end{align*}
characterized by restricting to multiplication with $(2\pi i)^k$ on $MU_{2k} \to MU_{2k}\otimes \C$. 
Let 
$$
K = 1+K_2(\sigma_1)+K_4(\sigma_1,\sigma_2)+ \cdots
$$ 
be the multiplicative sequence satisfying $\phi([f])=f_*K(N_f)$. 
For $p\in \Z$ we set $K^p=(2\pi i)^p \cdot K$ and
\begin{align*}
\phi^p([f])=f_*K^p(N_f).
\end{align*}
%
Let $f\colon Z \to X$ be a proper complex-oriented map, 
and let $\nabla_f$ be a connection on $N_f$. 
By Chern--Weil theory there is a well-defined form $c(\nabla_f)\in \Ah^*(Z)$ representing the total Chern class $c(N_f)$. 
In fact, with respect to local coordinates, we have 
\[
c(\nabla_f)=1+c_1(\nabla_f)+ c_2(\nabla_f) + \cdots = \det\left(I-\frac{1}{2\pi i}F^{\nabla_f}\right)
\]
where $F^{\nabla_f}$ denotes the curvature of $\nabla_f$. 
Then the form 
\[
K(\nabla_f):=K(c_1(\nabla_f),c_2(\nabla_f),\dots)\in\Ah^0(Z;\Vh_*)
\]
represents the cohomology class $K(N_f)$.

\begin{defn}
For a geometric cycle $\geocycle\in \geocycles^n(X)$ we define, using the orientation of $f$ induced by its complex orientation, the current 
\begin{align*}
\phi^p(\geocycle)=f_*K^p(\nabla_f)\in \Ds^n(X;\Vh_*).
\end{align*}
\end{defn}

Note that $\phi^p(\geocycle)$ is a closed current representing the cohomology class $\phi^p([f])=f_*K^p(N_f)\in H^n(X;\Vh_*)$. 
By de Rham's work \cite[Theorem 14]{deRham} we can always find a current $h\in \Ds^{n-1}(X;\Vh_*)$ such that 
\begin{align*}
\phi^p(\geocycle) - dh =  f_*K^p(\nabla_f) - dh ~ \text{is a form, i.e., lies in} ~ \Ah^n(X;\Vh_*).    
\end{align*} 

\begin{defn}\label{def:ZMUn(p)(X)}
Let $X$ be a complex manifold and $n, p$ integers. 
The group of \emph{Hodge filtered cycles} of degree $(n,p)$ on $X$ is defined as the subgroup 
\[
ZMU^{n}(p)(X)\subset \left(\geocycles^n(X)\times \Ds^{n-1}(X;\Vh_*)/d\Ds^{n-2}(X;\Vh_*)\right)
\]
consisting of pairs $\hfcycle=(\geocycle, h)$ satisfying
\begin{align*}
f_*K^p(\nabla_f) - dh \in F^p\Ah^n(X;\Vh_*).
\end{align*}
\end{defn}

\begin{remark}\label{rem:HFC_cycles_are_triples}
To simplify the notation, we will often write $\phi$ and $K$ instead of $\phi^p$ and $K^p$, respectively.
We may sometimes consider a Hodge filtered cobordism cycle as a triple 
\[
\hfcycle = (\geocycle, \omega, h)\in \geocycles^n(X)\times F^p\Ah^n(X;\Vh_*)\times \Ds^{n-1}(X;\Vh_*)/d\Ds^{n-2}(X;\Vh_*)
\] 
where $(\geocycle, h)\in ZMU^n(p)(X)$ and the form $\omega := \phi(\geocycle)- dh =  f_*K(\nabla_f) - dh$. 
\end{remark}

Next we introduce the cobordism relation.  
The group of geometric bordism data over $X$ is the subgroup of elements $\widetilde b \in \geocycles^n(\R\times X),$ with underlying maps $b=(c_b,f)\colon W\to \R\times X$ such that $0$ and $1$ are regular values for $c_b$.
Then $W_t=c_b^{-1}(t)$ is a closed manifold for $t=0,1$, and 
$f_t=f|_{W_t}$ is a geometric cycle. 
We define
\begin{align*}
\partial\widetilde b := \geocycle_1-\geocycle_0 \in \geocycles^n(X)
\end{align*}
and, setting $W_{[0,1]}=c_b^{-1}([0,1])$, we define the current 
\begin{align}\label{eq:GeometricBordismDatum}
\psi^p(\widetilde b) := 
(-1)^{n}\left(f|_{W_{[0,1]}}\right)_* \left(K^p(\nabla_b)\right). 
\end{align}
%
We will often write $\psi$ instead of $\psi^p$ to simplify the notation. 
By \cite[Proposition 2.17]{ghfc}, a geometric bordism datum $\geocob$  over $X$ satisfies 
\begin{align*}
\phi^p(\partial \geocob) - d\psi^p(\geocob) = 0. 
\end{align*}
Hence we consider $(\partial \geocob,\psi^p(\geocob))$ as a Hodge filtered cycle of degree $(\codim b,p)$. 
We call such cycles \emph{nullbordant} and let $BMU_{\geo}^n(p)(X)\subset ZMU^n(p)(X)$ denote the subgroup they generate.
We follow Karoubi in \cite[\S 4.1]{karoubi43} and denote 
\begin{align}
\label{eq_def_f_tilde}
\widetilde F^p\Ah^{n-1}(X;\Vh_*):= F^p\Ah^{n-1}(X;\Vh_*) + d\Ah^{n-2}(X;\Vh_*).    
\end{align}
We define the map 
\begin{align}\label{eq:a_forms_def}
a \colon d^{-1}\left(F^p\Ah^{n}(X;\Vh_*)\right)^{n-1} \rightarrow ZMU^n(p)(X),\quad a(h) := (0,h)
\end{align} 
where $d^{-1}\left(F^p\Ah^{n}(X;\Vh_*)\right)^{n-1}$ denotes the subset of elements in $\Ah^{n-1}(X;\Vh_*)$ which are sent to the subgroup  $F^p\Ah^n(X;\Vh_*)$ under $d \colon \Ah^{n-1}(X;\Vh_*) \to \Ah^{n}(X;\Vh_*)$.  
%
The group of \emph{Hodge filtered cobordism relations} is defined as 
\begin{align*}
BMU^n(p)(X)=BMU^n_{\geo}(p)(X) + a\left(\widetilde F^p\Ah^{n-1}(X;\Vh_*)\right).
\end{align*}


\begin{defn}\label{def:geometric_HFC_bordism_MUn(p)(X)}
Let $X \in \Manc$ and let $n$ and $p$ be integers. 
The \emph{geometric Hodge filtered cobordism} group of $X$ of degree $(n,p)$ is defined as the quotient 
\[
MU^{n}(p)(X):=\frac{ZMU^{n}(p)(X)}{BMU^n(p)(X)}.
\]
We denote the Hodge filtered cobordism class of the cycle $\gamma = (\geocycle,h) = (f,N_f,\nabla_f,h)$ 
by $[\gamma] = [\geocycle,h] = [f,N_f,\nabla_f,h]$.
\end{defn}


We define maps $R$ and $I$ on the level of cycles as follows: 
\begin{align}\label{cyclstructmaps}
R \colon ZMU^n(p)(X) & \to F^p\Ah^n(X;\Vh_*)_{\cl}, 
& R(\geocycle ,h) & = f_*K(\nabla_f) - dh \\ 
\nonumber I \colon ZMU^n(p)(X) & \to ZMU^n(X), 
& I(\geocycle ,h) & = f
\end{align}
Note that the maps $R$, $I$, and $a$ above induce well-defined homomorphisms on cohomology by \cite[Proposition 2.19]{ghfc}. 

\begin{remark}\label{rem:motivation_for_R}
Note that $[R(\hfcycle)]=\phi(I(\hfcycle))$. 
In that sense, $R$ refines the topological information of $I$ with Hodge filtered differential geometric content. 
It is shown in \cite{ghfc} that $R$ and $I$ fit in a homotopy pullback in a suitable model category which can be used to construct Hodge filtered cobordism. 
\end{remark}

For the following theorem we let $\overline{\phi}$ denote the composition of $\phi$ with the homomorphism induced by reducing the coefficients modulo $F^p$. 
\begin{theorem}
\label{thm:main_properties_HFC}
For every $p\in \Z$, the assignment $X \mapsto MU^*(p)(X)$ has the following properties: 
\begin{itemize}
    \item For every $X\in \Manc$ there is the following long exact sequence: 
\begin{align*}
    \xymatrix @R=0pc{
    \cdots\ar[r] & H^{n-1}\left(X;\frac{\Ah^*}{F^p}(\Vh_*)\right) \ar[r]^-a & MU^n(p)(X)\ar[r]^-I& \\
    MU^n(X)\ar[r]^-{\overline{\phi}}& H^{n}\left(X;\frac{\Ah^*}{F^p}(\Vh_*)\right) \ar[r]^-a & MU^{n+1}(p)(X)\ar[r]&\cdots
    }
\end{align*}
    \item For every holomorphic map $g \colon Y \to X$ and every $n$ there is a homomorphism 
\begin{align*}
g^*\colon MU^n(p)(X)\to MU^n(p)(Y).
\end{align*}
Hence $MU^n(p)$ is a contravariant functor on $\Manc$.
    \item For every $X\in \Manc$, there is a structure of a bigraded ring on 
    \begin{align*}
    MU^*(*)(X)=\bigoplus_{n,p}MU^n(p)(X). 
    \end{align*}
\end{itemize}
\end{theorem}
\begin{proof}
The first assertion is proven in  \cite[\S 2.6]{ghfc} and follows from a direct verification of the exactness. 
The second and third assertions are proven in \cite[\S 2.7]{ghfc} and \cite[\S 2.8]{ghfc}, respectively. 
We will, however, recall the construction of the pullback and of the ring structure in section \ref{sec:pushforward_along_MUDOrientations}. 
\end{proof}

For later purposes we now show how the Hodge filtered cobordism class depends on the connection on the representative of the normal bundle. 

\begin{defn}
\label{def:Chern-Simons_form_ses}
Let $X$ be a smooth manifold, and let 
\begin{align*}
\Eh =\left(\xymatrix{0\ar[r] & E_1\ar[r]& E_2\ar[r] & E_3\ar[r] & 0}\right)
\end{align*}
be a short exact sequence of complex vector bundles over $X$ with connections $\nabla^{E_i}$ on $E_i$. 
Let $\pi \colon [0,1]\times X \to X$ denote the projection. 
Let $\nabla^{\pi^*E_2}$ be a connection on $\pi^*E_2$ which equals $\pi^*\nabla^{E_2}$ near $\{1\} \times X$ and equals $\pi^*(\nabla^{E_1} \oplus \nabla^{E_3})$ near $\{0\} \times X$. 
The \emph{Chern--Simons transgression form} of the short exact sequence $\Eh$ associated to the multiplicative sequence $K$ is given by  
\begin{align*}
\KCS(\Eh) = \KCS(\nabla^{E_1},\nabla^{E_2},\nabla^{E_3}) = \int_{[0,1]\times X/X} K(\nabla^{\pi^*E_2}) \in \Ah^{-1}(X;\Vh_*)/\Imm(d). 
\end{align*} 
\end{defn}

\begin{remark}
The construction of $\KCS(\Eh)$ requires choosing a section $s \colon E_3 \to E_2$ as well as a connection $\nabla^{\pi^*E_2}$.  
However, the form $\KCS(\Eh)$ is independent of these choices in the quotient $\Ah^{-1}(X;\Vh_*)/\Imm(d)$. 
By Stokes' theorem, the derivative of the Chern--Simons form $\KCS(\nabla^{E_1},\nabla^{E_2},\nabla^{E_3})$ satisfies 
\begin{align*}
d\KCS\left(\nabla^{E_1},\nabla^{E_2},\nabla^{E_3}\right) &= K\left(\nabla^{E_2}\right) - K\left(\nabla^{E_1}\oplus\nabla^{E_3}\right)\\
        &=K(\nabla^{E_2})-K\left(\nabla^{E_1}\right)\wedge K\left(\nabla^{E_3}\right).
\end{align*}
\end{remark}

\begin{remark}
\label{rem:Chern-Simons_form}
We will often consider the following special case. 
Let $E$ be a complex vector bundle over the smooth manifold $X$.  
Let $\nabla_0$ and $\nabla_1$ be two connections on $E$. We can form the short exact sequence
\begin{align*}
\Eh =\left(\xymatrix{0\ar[r] & E\ar[r]^{id}& E\ar[r] & 0\ar[r] & 0}\right)
\end{align*}
and define $\KCS(\nabla_0,\nabla_1) := \KCS(\Eh)$. This Chern--Simons transgression form can be expressed as
\begin{align*}
\KCS(\nabla_1,\nabla_0) = \int_{[0,1]\times X/X} K(t\cdot \pi^*\nabla_1 + (1-t)\cdot \pi^*\nabla_0)  
\end{align*} 
and its derivative satisfies
\begin{align*}
d\KCS(\nabla_1,\nabla_0) = K(\nabla_1) - K(\nabla_0).   
\end{align*}
\end{remark}

\begin{lemma}
\label{bordismCycleDependenceOnConnection}
Let $\geocycle_0=(f,N,\nabla_0),\text{ and } \geocycle_1=(f,N,\nabla_1)\in \geocycles^n(X)$ be two geometric cycles over $X$ with the same underlying complex-oriented map $f \colon Z \to X$. 
Then there is a geometric bordism $\geocob$ with $\partial\geocob = \geocycle_1-\geocycle_0$ and 
\begin{equation*}
    \psi(\geocob) = (-1)^nf_*\KCS(\nabla_0,\nabla_1).
\end{equation*}
\end{lemma}

\begin{proof}
Let $b=\id_\R\times f \colon \mathbb R\times Z\to \mathbb R\times X$, and let $\pi_Z \colon \mathbb R\times Z\to Z$ denote the projection. With the product complex orientation, using $N_{\id_\R}=0$, we have $N_b=N_{\id_\R\times f}=\pi_Z^*N_f$. On $\pi_Z^*N_f$ we consider the connection 
\[
\nabla_b:=t\cdot \pi_Z^*\nabla_0 + (1-t)\cdot \pi_Z^*\nabla_1
\] 
where $t$ is the $\R$-coordinate. 
We can then promote $b$ to a geometric bordism $\geocob=(b,N_b,\nabla_b)$. 
Then we have $\partial \geocob = \geocycle_1-\geocycle_0$. 
Using 
$f\circ \pi_Z = \pi_X\circ b$ the assertion follows from
\begin{align*}
f_*\KCS(\nabla_0,\nabla_1)& = f_*\circ(\pi_Z|_{[0,1]\times Z})_*K(\nabla_b)\\
   & = \left((\pi_X\circ b)|_{[0,1]\times Z}\right)_*K(\nabla_b). \qedhere 
\end{align*}
\end{proof}


\subsection{Currential Hodge filtered complex cobordism}\label{sec:currentialHFC}

Now we introduce a new and alternative description of Hodge filtered cobordism groups by considering the Hodge filtration on currents instead of forms. 
The difference to the previous definition may seem negligible but turns out to be crucial for the construction of a general pushforward later. 

\begin{defn}\label{def:currential_cycles}
Let $X$ be a complex manifold and $n, p$ integers. 
We define the group of \emph{currential} Hodge filtered cycles $ZMU^n_\delta(p)(X)$ as the subgroup 
\[
ZMU^n_\delta(p)(X) \subset \left(\geocycles^n(X) \times \Ds^{n-1}(X;\Vh_*)/d\Ds^{n-2}(X;\Vh_*)\right)
\]
consisting of pairs $(\geocycle, h)$
such that 
\begin{align}\label{eq:def_currential_condition_on_current}
\phi(\geocycle) - dh =  f_*K(\nabla_f) - dh \in F^p\Ds^n(X;\Vh_*).
\end{align}
\end{defn}
We will sometimes write a currential Hodge filtered cycle $(\geocycle, h)$ as a triple $(\geocycle, T, h)$ with $T=\phi(\geocycle)-dh$. 
Let $a_{\delta}$ denote the map 
\begin{align*}
a_{\delta} \colon d^{-1}\left(F^p\Ds^n(X;\Vh_*)\right)^{n-1} \rightarrow ZMU^n(p)(X),\quad a_{\delta}(h) := (0,h), 
\end{align*} 
where $d^{-1}\left(F^p\Ds^n(X;\Vh_*)\right)^{n-1}$ denotes the subset of elements in $\Ds^{n-1}(X;\Vh_*)$ which are sent to the subgroup  $F^p\Ds^n(X;\Vh_*)$ under $d \colon \Ds^{n-1}(X;\Vh_*) \to \Ds^{n}(X;\Vh_*)$. 
We define the group of \emph{currential cobordism relations} by 
\[
BMU_{\delta}^n(p)(X) := BMU^n_{\geo}(X)+a_{\delta}\left(\widetilde F^p\Ds^*(X;\Vh_*)\right).
\]

\begin{defn} 
For $X\in \Manc$ and integers $n$, $p$, 
we define the \emph{currential Hodge filtered cobordism} groups by
\begin{align*}
MU_{\delta}^{n}(p)(X):=ZMU_{\delta}^{n}(p)(X)/BMU_{\delta}^{n}(p)(X).
\end{align*}
\end{defn}

Similar to \eqref{cyclstructmaps}, we define maps on the level of currential cycles as follows: 
\begin{align}\label{cyclstructmaps_currents}
R_{\delta} \colon ZMU_{\delta}^n(p)(X) & \to F^p\Ds^n(X;\Vh_*), & ~ ~ R_{\delta}(\geocycle,h) & = f_*K(\nabla_f) - dh \\ 
\nonumber I_{\delta} \colon ZMU_{\delta}^n(p)(X) & \to ZMU^n(X),  & ~ ~ I_{\delta}(\geocycle, h) & = f. 
\end{align}
By slight abuse of notation, we also denote by the symbols $R_{\delta}$, $I_{\delta}$ and $a_{\delta}$ the corresponding induced homomorphisms on cohomology groups. 
When the context is clear, we will often drop the subscript $\delta$ from the notation. 
For the next statement let $\overline{\phi}_{\delta}$ denote the composition of $\phi$ 
with the homomorphism induced by reducing the coefficients modulo $F^p$.

\begin{prop}
\label{currentiallongexactseq}
Let $X$ be a complex manifold. 
There is a long exact sequence  
\begin{align} 
    \xymatrix @R=0pc{
    \cdots\ar[r]^-{\overline{\phi}_{\delta}} & H^{n-1}\left(X;\frac{\Ds^*}{F^p}(\Vh_*)\right) \ar[r]^-{a_{\delta}} & MU_{\delta}^n(p)(X)\ar[r]^-{I_{\delta}} & \\
    MU^n(X)\ar[r]^-{\overline{\phi}_{\delta}} & H^{n}\left(X;\frac{\Ds^*}{F^p}(\Vh_*)\right) \ar[r]^-{a_{\delta}}& MU_{\delta}^{n+1}(p)(X)\ar[r]^-{I_{\delta}} &\cdots
    }
\end{align}
\end{prop}
\begin{proof}
The proof follows that of \cite[Theorem 2.21]{ghfc} closely. 
We provide the details of the proof for the convenience of the reader. 
We start with exactness at $MU_{\delta}^n(p)(X)$. 
By definition of $a_{\delta}$ and $I_{\delta}$ we have  
\[
I_{\delta}(a_{\delta}([h]))=I_{\delta}([0,dh,h])=0.
\]
To show the converse we work at the level of cycles. 
Let $\hfcycle=(\geocycle,h)\in ZMU_{\delta}^n(p)(X)$ and suppose $I_{\delta}(\gamma)=0$. 
That means $f=\partial b$ for some bordism datum $b$. 
We may extend the geometric structure of $\geocycle$ over $b$ and obtain a geometric bordism datum $\widetilde b$ such that $\partial \widetilde b=\geocycle$. 
We then have 
\[
(\geocycle,h)-(\partial\widetilde b,\psi(\widetilde b)) = (0,h') 
= a_{\delta}(h').
\]
The last equality follows from the observation that, since $(0,h')\in ZMU^n(p)(X)$ is a currential Hodge filtered cycle, we must have
$dh' \in F^p\Ds^n(X;\Vh_*)$. 
Hence we know $\gamma \in BMU_{\delta}^n(p)(X)$. 

Next we show exactness at $MU^n(X)$. 
The vanishing $\overline{\phi}_{\delta}\circ I_{\delta} = 0$ follows from the following commutative diagram, where the bottom row is exact:
\begin{align*}
\xymatrix{
 MU_{\delta}^n(p)(X)\ar[d]_-{R_{\delta}} \ar[r]^-{I_{\delta}} & MU^n(X) \ar[d]_-{\phi} \ar[dr]^-{\overline{\phi}_{\delta}} & \\
H^{n}(X;F^p\Ds^*(\Vh_*))\ar[r]^-{\inc_*} & H^n(X;\Vh_*)\ar[r] & H^{n}\left(X;\frac{\Ds^*}{F^p}(\Vh_*)\right).
}
\end{align*}
Conversely, suppose $\overline\phi_{\delta}([f])=0$. 
Then we can find $\widetilde{\omega} \in F^p\Ds^n(X;\Vh)$ such that 
\[
\phi([f]) = \inc_*([\widetilde{\omega}]).
\]
Let $\nabla_f$ be a connection on $N_f$ so that we get a geometric cycle $\geocycle$ with $I_{\delta}(\geocycle)=f$. 
Then $\phi(\geocycle)$ is a current representing $\phi([f])$. 
Hence $\phi(\geocycle)$ and $\widetilde{\omega}$ are cohomologous, 
i.e., there is a current $h\in \Ds^{n-1}(X;\Vh_*)$ such that $\phi(\geocycle)- dh = \widetilde{\omega}$. 
Then $\hfcycle:=(\widetilde f,h)$ is a currential Hodge filtered cycle with $I_{\delta}(\hfcycle) = f$.

Now we show exactness at $H^{n}\left(X;\frac{\Ds^*}{F^p}(\Vh_*)\right)$. 
Let $f\colon Z\to X$ be a bordism cycle on $X$. 
We will show $a_{\delta}(\overline\phi_{\delta}([f])) = 0 \in MU_{\delta}^{n+1}(p)(X)$. 
Lifting $f$ to a geometric cycle $\geocycle\in\geocycles^n(X)$ we can write
\[
a_{\delta}(\overline\phi_{\delta}([f])=[0,\phi(\geocycle)].
\]
We may build from $\geocycle$ a geometric bordism datum $\geocob$ with underlying map
\[
\xymatrix{Z\ar[r]^-{(\frac{1}{2}, f)}&\mathbb R\times X}
\]
where $\frac{1}{2}$ denotes the constant map with value $\frac{1}{2}$.  
Clearly $\partial\geocob =0$. 
Moreover, we have $\psi(\geocob)=(-1)^{n}\phi(\geocycle)$. 
Hence
\[
(\partial \geocob, \psi(\geocob)) = (0,(-1)^{n} \phi(\geocycle)) \in BMU^n_{\delta}(p)(X)
\]
and we conclude that $a_{\delta}(\overline\phi_{\delta}([f]))=0$.

Conversely, suppose that $h \in (d^{-1}F^p\Ds^n(X;\Vh_*))^{n-1}$ is such that $a_{\delta}(h)=(0,h)$ represents $0$ in $MU_{\delta}^n(p)(X)$. 
Then there is a geometric bordism datum $\geocob$ with underlying map $(c_b,f_b)\colon W\to \mathbb R\times X$, and a current $h' \in \widetilde F^p\Ds^{n-1}(X;\Vh_*)$ such that
\[
(0,h) = (\partial\geocob,\psi(\geocob)+h').
\]
Note that $\widetilde F^p\Ds^{n-1}(X;\Vh_*)$ is the group of relations for $H^{n-1}\left(X;\frac{\Ds^*}{F^p}(\Vh_*)\right)$, where we therefore have
$$
[h]=[\psi(\geocob)].
$$
Since $\partial\geocob=0$, we have  
\[
f:=f_b|_{c_b^{-1}([0,1])}\in ZMU^n(X)
\]
is a bordism cycle. 
By definition of $\psi$, we have $\psi(\geocob)=(-1)^n \phi(\geocycle)$ where $\geocycle$ is the obvious geometric cycle over $f$. Hence
$$
[h]=[\psi(\geocob)]=(-1)^n\overline\phi_\delta([f]) \in \Imm(\overline{\phi}_{\delta}).
$$
This finishes the proof.
\end{proof}

There is a natural homomorphism 
\[
\tau \colon ZMU^n(p)(X)\to ZMU_{\delta}^n(p)(X), ~ 
(\geocycle, h)\mapsto (\geocycle, h)
\]
which forgets that $\phi(\geocycle)- dh$ is a form and not just a current. 
Since $\tau$ sends $BMU^n(p)(X)$ to  $BMU_{\delta}^n(p)(X)$, it follows that that there is an induced natural homomorphism 
\begin{align*}
\tau \colon MU^n(p)(X)\to MU_{\delta}^n(p)(X).
\end{align*}

\begin{theorem}
\label{equivGeoDesc}
For every $X\in \Manc$ and all integers $n$ and $p$, 
the natural homomorphism $\tau \colon MU^n(p)(X)\to MU_{\delta}^n(p)(X)$ is an isomorphism. 
\end{theorem}
\begin{proof}
The long exact sequences of Theorem \ref{thm:main_properties_HFC} and Proposition \ref{currentiallongexactseq}  fit into the commutative diagram
\[
\xymatrix{
\cdots\ar[r] & H^{n-1}\left(X;\frac{\Ah^*}{F^p}(\Vh_*)\right) \ar[d]_-{\cong} \ar[r]^-a & MU^n(p)(X)\ar[d]\ar[r]^-I & MU^n(X)\ar[d]^-{\id} \ar[r]^-{\overline{\phi}} \ar[d] &  \cdots  \\
\cdots \ar[r] &  H^{n-1}\left(X;\frac{\Ds^*}{F^p}(\Vh_*)\right) \ar[r]^-{a_{\delta}} & MU_{\delta}^n(p)(X)\ar[r]^-{I_{\delta}}& MU^n(X)\ar[r]^-{\overline{\phi}_{\delta}} &\cdots
}
\] 
That the left-most vertical arrow is an isomorphism, is a corollary of the fact that the Dolbeault--Grothendieck lemma holds for currents as well as forms, see Lemma \ref{lem:DolbeaultGrothendieck} for the full argument. The right-most arrow is the identity. 
The assertion now follows from the five-lemma.
\end{proof}

\begin{remark} \label{remark:HodgeFiltrationCurrents}
In \cite{ghfc}, $MU^n(p)(-)$ is defined on the larger category $\ManF$ of manifolds with a filtration of $\Ah^*$. 
We note that Theorem \ref{equivGeoDesc} does not extend to that context, since its proof uses that there is a Hodge filtration for currents which extends that of forms and that the inclusion $F^p\Ah^*\to F^p\Ds^*$ is a quasi-isomorphism. 
\end{remark}

\begin{remark}\label{rem:difference_smooth_and_currential_pushforward}
As we discussed in the introduction, Theorem \ref{equivGeoDesc} reflects an important difference between Hodge filtered cohomology and differential cohomology. 
\end{remark}


\section{Hodge filtered \texorpdfstring{$MU$}{MU}-orientations} \label{sec:MUDorientations}

We will now define the notion of a Hodge filtered $MU$-orientation of a holomorphic map in two steps:  
First as a type of Hodge filtered $K$-theory class with $\Vh_*$-coefficients. 
Then we apply this to the normal bundle of a holomorphic map. 
Recall from \eqref{eq_def_f_tilde} the notation 
\[
\widetilde F^0\Ah^{-1}(X;\Vh_*)= F^0\Ah^{-1}(X;\Vh_*) + \Imm(d) \subset \Ah^{-1}(X;\Vh_*).
\]

\begin{defn}\label{def:MUDorientation}
Let $X$ be a complex manifold. 
We define the group $\MUDors(X)$ of \emph{Hodge filtered $MU$-orientations}, 
\emph{$\MUD$-orientations} for short, 
to be the quotient of the free abelian group generated by triples $\epsilon:=(E,\nabla, \sigma)$ where $E$ is a complex vector bundle on $X$, $\nabla$ is a connection on $E$ and $\sigma\in \Ah^{-1}(X;\Vh_*)/\widetilde F^0\Ah^{-1}(X;\Vh_*)$ 
such that 
\begin{align}\label{eq:notation_K_epsilon}
\KK(\epsilon) = \KK(E,\nabla,\sigma) := K(\nabla)-d\sigma\in F^0\Ah^0(X;\Vh_*)_{\cl} 
\end{align}
is a form in filtration step $F^0$ 
modulo the subgroup generated by $(\trivbc_X,d,0)$ for the trivial bundle $\trivbc_X$ on $X$ with the canonical connection and by 
\begin{align*}
(E_2,\nabla_2, \sigma_2) -  (E_1,\nabla_1, \sigma_1) - (E_3,\nabla_3, \sigma_3)
\end{align*}
whenever there is a short exact sequence of the form 
\begin{align*}
\xymatrix{0\ar[r] & E_1\ar[r]& E_2\ar[r] & E_3\ar[r] & 0}
\end{align*}
with the identity 
\begin{align}\label{eq:defMUors}
\sigma_2=\sigma_1\wedge \KK(\epsilon_3)+ \sigma_3\wedge K(\nabla_1)+\KCS(\nabla_1,\nabla_2,\nabla_3)
\end{align}
in $\Ah^{-1}(X;\Vh_*)/\widetilde F^0\Ah^{-1}(X;\Vh_*)$.  
We denote the image of $(E,\nabla,\sigma)$ in the quotient by $[E,\nabla,\sigma]$. 
\end{defn} 


\begin{remark}\label{rem:orientation_of_filtration_q}
For $q\in \Z$ we could modify the above definition and define an $\MUD$-orientation of \emph{filtration $q$} to be a triple $(E,\nabla,\sigma)$ as above such that $\KK(\epsilon) = K(\nabla)-d\sigma\in F^q\Ah^0(X;\Vh_*)$. 
Using appropriate relations, the addition on $\MUDors(X)$ extends to the direct limit of pointed sets over all $q\in \Z$.
The group $\MUDors(X)$ of Definition \ref{def:MUDorientation} is the subgroup of orientations of filtration $0$. 
Since we do not know of applications to support the additional generality and complexity, we only consider orientations of filtration $0$ in this paper. 
\end{remark}


We will now discuss the group $\MUDors(X)$ in more detail. 

\begin{lemma}
\label{lem:formula_for_addition_in_mudors}
The addition in $\MUDors(X)$ is given by
$$
[\epsilon_1] + [\epsilon_2] = [E_1,\nabla_1,\sigma_1] +[E_2,\nabla_2,\sigma_2] 
= [E_1\oplus E_2, \nabla_1\oplus\nabla_2, \sigma_{12}]
$$
where
\begin{align*}
\sigma_{12}= \sigma_1\wedge \KK(\epsilon_2) + \sigma_2\wedge \left(K(\nabla_1)\right).
\end{align*}
\end{lemma}
\begin{proof}
We consider the sequence
$$
\xymatrix{0\ar[r]& E_1\ar[r]& E_1\oplus E_2 \ar[r]& E_2\ar[r]&0}
$$
Since the connections split, $\KCS(\nabla_1,\nabla_1\oplus\nabla_2,\nabla_2)=0$. 
Hence condition \eqref{eq:defMUors} for 
$$
(E_1\oplus E_2, \nabla_1\oplus\nabla_2, \sigma_{12})-(E_1,\nabla_1,\sigma_1)-(E_2,\nabla_2,\sigma_2)
$$
to be a relation reduces to
\begin{equation*}
\sigma_{12} = \sigma_1 \wedge \KK(\epsilon_2) + \sigma_2 \wedge K(\nabla_1). \qedhere
\end{equation*}
\end{proof}

\begin{remark}\label{rem:identity_in_KMUD}
It follows from Lemma \ref{lem:formula_for_addition_in_mudors} that the identity element of $\MUDors(X)$ is represented by the triple $(0,d,0)$ where the first $0$ denotes the zero-dimensional trivial bundle. 
\end{remark}


\begin{remark}
\label{rem:equivalent_mud_orientations_over_same_bundle}
Suppose we have two triples $\MUDorrep_1=(E,\nabla_1,\sigma_1)$ and $\MUDorrep_2=(E,\nabla_2,\sigma_2)$ with the same underlying bundle $E$. 
By Remark \ref{rem:identity_in_KMUD}, the triple $0=(0,d,0)$ represents the identity element in $\MUDors(X)$. 
Since $K(d)=1$, considering $E \xto{\id} E$ as a short exact sequence as in Remark \ref{rem:Chern-Simons_form}, we get a relation $[\MUDorrep_2]-[\MUDorrep_1]-[0]$ for $\MUDors$, i.e., we have $[\epsilon_1] = [\epsilon_2]$ in $\MUDors(X)$, 
if and only if
\begin{align*}
\sigma_2 = \sigma_1+\KCS(\nabla_{1},\nabla_{2}).
\end{align*}
\end{remark}


\begin{remark}\label{rem:addition_is_commutative}
There is a hidden symmetry in Equation \eqref{eq:defMUors} of Definition \ref{def:MUDorientation}: 
Since $\sigma_1$ and $\sigma_3$ are of odd degree, we have 
\begin{align*}
\sigma_1 \wedge d\sigma_3 = \sigma_3 \wedge d\sigma_1 \quad \mathrm{mod\ Im}(d). 
\end{align*}
Hence, modulo $\Imm(d)$, we have 
\begin{align*}
\sigma_{1} \wedge (K(\nabla_3) -d\sigma_3) + \sigma_3 \wedge K(\nabla_1) 
& = \sigma_1 \wedge K(\nabla_3) - \sigma_1 \wedge d\sigma_3 + \sigma_3 \wedge K(\nabla_1) \\
& = \sigma_1 \wedge K(\nabla_3) + \sigma_3 \wedge K(\nabla_1) - \sigma_3 \wedge d\sigma_1 \\
& = \sigma_1 \wedge K(\nabla_3) + \sigma_3 \wedge (K(\nabla_1)-d\sigma_1). 
\end{align*}
Using the map $\KK$ we can rewrite this relation as 
\begin{align*}
\sigma_{1} \wedge \KK(\epsilon_3) + \sigma_3 \wedge K(\nabla_1) 
= \sigma_1 \wedge K(\nabla_3) + \sigma_3 \wedge \KK(\epsilon_1) \quad \mathrm{mod\ Im}(d). 
\end{align*}
\end{remark}


We now discuss further properties of the assignment 
\[
\epsilon = (E,\nabla, \sigma) \mapsto \KK(\epsilon) = K(\nabla)-d\sigma  
\]
defined in \eqref{eq:notation_K_epsilon}. 
We note that there is a certain similarity in the behaviors and roles of the forms $R(\geocycle, h)$ (respectively current $R_\delta(\geocycle, h)$) and $\KK(\epsilon)$. 
Both $R_\delta(\geocycle, h)$ and $\KK(\epsilon)$ contribute to the construction of the pushforward along holomorphic maps in section \ref{sec:pushforward_along_MUDOrientations} (see also Lemma \ref{pushforward_current_wedge}, Remark \ref{rem:assymetryInPushforward}, Proposition \ref{prop:pushforwards_commute} and Remark \ref{rem:current_of_pushforward_formula}). 
Moreover, while the current $f_*K(\nabla_f)$ is not a cobordism invariant, the class of the difference $R(\geocycle, h) = f_*K(\nabla_f) - dh$ is indeed invariant. 
Similarly, we will show in Proposition \ref{prop:K_is_monoid_morphism} that $\KK(\epsilon) = K(\nabla)-d\sigma$ is an invariant of the equivalence class $[\epsilon]$ in $\MUDors(X)$ while $K(\nabla)$ is not. 
In fact, we will show that $\KK$ respects the group structure on $\MUDors(X)$. 
This result will be used in several of our main results and their proofs. 
In particular, $\KK$ plays a key role in the proof of the functoriality of the pushforward in Theorem \ref{thm:pushforward_is_functorial}.

\begin{lemma}\label{lemma:K_is_additive}
For representatives $(E_i,\nabla_i,\sigma_i)$ with $i=1,2$ of generators in $\MUDors(X)$ we have  
\begin{equation*}
\KK(\MUDorrep_1+\MUDorrep_2) = \KK(\MUDorrep_1)\wedge\KK(\MUDorrep_2).
\end{equation*}
\end{lemma}
\begin{proof}
To prove the assertion we use Lemma \ref{lem:formula_for_addition_in_mudors}. 
Both $d\sigma_i$ and $K(\nabla_i)$ are closed and lie in $\Ah^0(X_i;\Vh_*)_{\cl}$. 
In particular, this means that they are in the  center of the ring $\Ah^*(X_i;\Vh_*)$.  
Using this fact we get
\begin{align*}
\KK(\MUDorrep_1+\MUDorrep_2) & = K(\nabla_1\oplus\nabla_2)-d\left(\sigma_1\wedge(K(\nabla_2)-d\sigma_2)+\sigma_2\wedge \left(K(\nabla_1)\right)\right) \\
& =  K(\nabla_1)\wedge K(\nabla_2) -d\sigma_1\wedge \KK(\MUDorrep_2) - d\sigma_2\wedge K(\nabla_1)\\
& = K(\nabla_1)\wedge \KK(\MUDorrep_2) -d\sigma_1\wedge \KK(\MUDorrep_2)\\
& = \KK(\MUDorrep_1)\wedge\KK(\MUDorrep_2). \qedhere
\end{align*}
\end{proof}

\begin{prop}\label{prop:K_is_monoid_morphism}
The map $\KK$ descends to a morphism of monoids 
\[
\KK \colon (\MUDors(X),+) \to (F^0\Ah^{0}(X;\Vh_*)_{\cl}, \wedge).
\]
\end{prop}
\begin{proof}
Since the triple $(\trivbc_X^N,d,0)$ represents the identity element in $\MUDors(X)$, we see that $\KK$ sends the identity element to the identity. 
The fact that $\KK$ descends to a map on $\MUDors(X)$ and respects the monoid structure then  follows from Lemma \ref{lemma:K_is_additive} and the defining relations of $\MUDors(X)$.  
\end{proof}


\begin{remark}\label{rem:solveforsigma}
Let $\MUDorrep = (E, \nabla,\sigma)$ in $\MUDors(X)$ be a representative of a generator in $\MUDors(X)$. We may consider $\KK(\epsilon) = K(\nabla) - d\sigma$ as a power series over the commutative ring $\Ah^{2*}(X)$ in the generators of $\Vh_*$. 
Having leading term $1$, $\KK(\epsilon)$ is an \emph{invertible} power series. 
\end{remark}

\begin{remark}\label{rem:inverse_of_orientation}
The triple $(\trivbc_X^N,d,0)$ represents the identity element in $\MUDors(X)$. 
Given a generator $\MUDorrep = (E,\nabla,\sigma)$ for $\MUDors(X)$, we can construct a class  $[\MUDorrep']$ such that $[\MUDorrep]+[\MUDorrep']=0$ in $\MUDors(X)$ as follows: 
Since $X$ is a finite-dimensional manifold,   
we can find a complex vector bundle $E'$ and an isomorphism $E\oplus E' \cong \trivbc^N_X$ for some $N$. 
We equip $E'$ with the connection $\nabla'$ induced from $d$ by the direct sum decomposition $E\oplus E'= \trivbc^N_X$. 
Using Remark \ref{rem:solveforsigma} we let
\[
\sigma' = -\sigma \wedge K(\nabla') \wedge \KK(\MUDorrep)^{-1}.
\]
To check that $\MUDorrep' = (E',\nabla',\sigma')$ satisfies $[\MUDorrep]+[\MUDorrep']=0$ we use Lemma \ref{lem:formula_for_addition_in_mudors} to write $[\MUDorrep]+[\MUDorrep'] = [\trivbc^N_X,d,\sigma'']$ where
\begin{align*}
\sigma'' & = \sigma\wedge K(\nabla')+\sigma'\wedge\KK(\MUDorrep) \\
 & = \sigma\wedge K(\nabla')- \sigma\wedge K(\nabla')\wedge \KK(\MUDorrep)^{-1}\wedge \KK(\MUDorrep)  \\
 & = 0 
\end{align*}
which proves the claim.  
\end{remark}

\begin{remark}\label{rem:orientation_on_two_of_three}
Assume we have a short exact sequence of complex vector bundles 
\begin{align*}
\xymatrix{0\ar[r] & E_1\ar[r]& E_2\ar[r] & E_3\ar[r] & 0}
\end{align*}
and that we have orientations $\MUDor_i = [E_i,\nabla_i,\sigma_i]$ involving two of the three bundles. 
Then it follows from the defining relations in $\MUDors(X)$ and Remark \ref{rem:solveforsigma} that for any connection $\nabla_j$ on the remaining bundle $E_j$ we can find a form $\sigma_j$ such that $\MUDor_j =[E_j,\nabla_j,\sigma_j]$ is an orientation and such that $\MUDor_1+\MUDor_3=\MUDor_2$. 
\end{remark}

\begin{defn}\label{def:orientation_pullback}
Let $f \colon X \to Y$ be a holomorphic map. 
Since the defining relations are compatible with pullbacks of bundles and connections, 
there is a well-defined pullback of orientations 
\[
f^* \colon \MUDors(Y) \to \MUDors(X)
\]
defined by $f^*[E,\nabla,\sigma] = [f^*E,f^*\nabla,f^*\sigma]$. 
\end{defn}


Next we define the notion of a Hodge filtered orientation of a holomorphic map. 
We will use this notion in the following section to define the pushforward along a holomorphic map. 
In section \ref{sec:canonical_MUD_orientation} we show that every holomorphic map has a canonical choice of a Hodge filtered orientation. 

\begin{defn}\label{def:orientation_of_a_map}
Let $g \colon X \to Y$ be a holomorphic map. 
We define a \emph{Hodge filtered $MU$-orientation of $g$}, or an $\MUD$-orientation of $g$ for short, to be a class 
\[
\MUDor=[N_g,\nabla_g,\sigma_g]\in \MUDors(X)
\]
where $N_g$ represents the complex stable normal bundle associated with $g$ as a complex-oriented map. 
\end{defn}

\begin{defn}\label{def:composed_orientation_on_maps}
Let $X_1 \xto{g_1} X_2 \xto{g_2} X_3$ be proper  holomorphic maps of complex codimension $d_1$ and $d_2$, respectively. 
Let $\MUDor_i \in \MUDors(X_i)$ be $\MUD$-orientations of $g_i$ for $i=1,2$.  
We define the \emph{composed Hodge filtered $MU$-orientation} of $g_2\circ g_1$ to be 
\[
\MUDor_1 + g_1^*\MUDor_2 \in \MUDors(X_1).
\]
\end{defn}

\begin{remark}
With the notation of Definition \ref{def:composed_orientation_on_maps}, we recall that the stable normal bundle of $g_2\circ g_1$ is isomorphic to the sum of the stable normal bundle of $g_1$ and the pullback of the stable normal bundle of $g_2$ along $g_1$. 
Hence $\MUDor_1 + g_1^*\MUDor_2$ is, in fact, a Hodge filtered $MU$-orientation of $g_2\circ g_1$ in the sense of Definition \ref{def:orientation_of_a_map}. 
\end{remark}


\section{Pushforward along proper Hodge filtered \texorpdfstring{$MU$}{MU}-oriented maps} \label{sec:pushforward_along_MUDOrientations}

We will now define a pushforward homomorphism for proper $\MUD$-oriented maps and show that it is functorial. 
Then we show that the pushforward is compatible with pullback and cup product. \\

Let $g \colon X \to Y$ be a holomorphic map and let $\MUDor$ be an orientation of $g$.  
We write 
$g^\MUDor$
for $g$ together with the orientation $\MUDor$ and refer to $g^\MUDor$ as an $\MUD$-oriented holomorphic map. 
If we want to specify the representative $\MUDorrep=(N_g,\nabla_g,\sigma_g)$ of $\MUDor$, 
we write $g^\MUDorrep=(g,N_g,\nabla_g,\sigma_g)$, and we write 
$\widetilde g^\MUDorrep$ for the underlying geometric cycle $\widetilde g^\MUDorrep=(g,N_g,\nabla_g)$. 
We will now define the pushforward of a Hodge filtered cycle along an oriented proper holomorphic map.

\begin{defn}\label{def:pushforward_on_cycles}
Let $g \colon X \to Y$ be a proper holomorphic map of complex codimension $d$. 
Let $\epsilon=(N_g,\nabla_g,\sigma_g)$ be a representative of an orientation class of $g$ in $\MUDors(X)$. 
If $\geocycle = (f,N_f,\nabla_f)$ is a geometric cycle on $X$, we write 
\[
\widetilde g^\MUDorrep \circ \geocycle= (g\circ f,\ N_f\oplus f^*N_g,\ \nabla_f\oplus f^*\nabla_g)
\]
for the composed geometric cycle on $Y$. 
We define the \emph{pushforward homomorphism} on currential Hodge filtered cycles by 
\begin{align}
g^\MUDorrep_*(\geocycle, h) = \left(\widetilde g^\MUDorrep \circ\geocycle,\  g_*\big( K(\nabla_g)\wedge h + \sigma_g \wedge R_{\delta}(\geocycle,h) 
\big) \right)
\end{align}
where $g_*$ denotes the pushforward of currents along $g$ and $R_{\delta}(\geocycle,h) = f_*K(\nabla_f) - dh$ is defined as in \eqref{cyclstructmaps_currents}. 
\end{defn}


We will explain the choices made in Definition \ref{def:pushforward_on_cycles} further in Remarks \ref{rem:choice_of_pushforward_formula} and \ref{rem:assymetryInPushforward} below. 
But first we need to check that the construction is well-defined, i.e., we have to show that $g^\MUDorrep_*(\geocycle, h)$ actually is a currential Hodge filtered cycle.  
We will achieve this in two steps as follows: 

\begin{lemma}\label{pushforward_current_wedge}
For every $(\geocycle,h) \in ZMU_{\delta}^n(p)(X)$ 
we have
\[
R_{\delta}(g^\MUDorrep_*(\geocycle, h)) = g_*(\KK(\epsilon) \wedge R_{\delta}(\geocycle,h)). 
\]
\end{lemma}
\begin{proof}
We check this claim by applying the definition of $R_{\delta}$ and then rewrite the current as follows: 
\begin{align*}
R_{\delta}(g^\MUDorrep_*(\geocycle, h)) & = g_*f_*K(\nabla_f\oplus f^*\nabla_g) - dg_*\left( K(\nabla_{g})\wedge h + \sigma_g \wedge R_{\delta}(\geocycle,h) \right) \\
& = g_*\left(f_*K(\nabla_f)\wedge K(\nabla_g)\right) - g_*\left(K(\nabla_ g)\wedge dh + d\sigma_g \wedge R_{\delta}(\geocycle,h) \right) \\
& = g_*\left( K(\nabla_g)\wedge (f_* K(\nabla_f) - dh ) - d\sigma_g \wedge R_{\delta}(\geocycle,h)\right) \\
& = g_*\left((K(\nabla_g)-d\sigma_g) \wedge R_{\delta}(\geocycle,h) \right) \\
& = g_*(\KK(\epsilon) \wedge R_{\delta}(\geocycle,h)). \qedhere
\end{align*}
\end{proof}

We can now use this observation to show that $g^\MUDorrep_*(\geocycle, h)$ is a currential Hodge filtered cycle: 

\begin{prop}\label{prop:pushforwardFiltrationStage}
For every $(\geocycle,h) \in ZMU_{\delta}^n(p)(X)$ 
we have 
\[
g^\MUDorrep_*(\geocycle,h) \in ZMU_{\delta}^{n+2d}(p+d)(Y).
\]
\end{prop}
\begin{proof}
It follows from the definition that $\widetilde g^\MUDorrep \circ\geocycle$ is a geometric cycle. 
Hence, by definition of currential Hodge filtered cycles in \ref{def:currential_cycles}, it remains to check that the current $R_{\delta}(g^\MUDorrep_*(\geocycle, h)) = \phi(\widetilde g^\MUDorrep \circ\geocycle) - dh$ satisfies condition \eqref{eq:def_currential_condition_on_current} on the filtration step of a current in a Hodge filtered cycle, 
i.e., we have to show that $R_{\delta}(g^\MUDorrep_*(\geocycle, h))$ lies in $F^{p+d}\Ds^{n+2d}(Y;\Vh_*)$. 
By Lemma \ref{pushforward_current_wedge} we know $R_{\delta}(g^\MUDorrep_*(\geocycle, h)) = g_*(\KK(\epsilon) \wedge R_{\delta}(\geocycle,h))$.  
Hence it suffices to observe that  $\KK(\epsilon) = K(\nabla_g)-d\sigma_g \in F^0\Ah^0(X;\Vh_*)$ and $R_{\delta}(\geocycle,h)) = f_*K(\nabla_f)-dh\in F^{p}\Ds^n(X;\Vh_*)$. 
Since $g$ is holomorphic of codimension $d$, 
it follows that 
\begin{align*}
g_*(\KK(\epsilon) \wedge R_{\delta}(\geocycle,h)) \in F^{p+d}\Ds^{n+2d}(Y;\Vh_*) 
\end{align*}
as required. 
\end{proof}


\begin{remark}\label{rem:choice_of_pushforward_formula} 
One might arrive at the formula for the current 
in the definition of $g^\MUDorrep_*(\geocycle, h)$ as follows:
If $\sigma_g=0$, then $g_*(K(\nabla_g)\wedge h)$ is the only natural candidate, and it does satisfy the desirable formulas. Having made that choice, consider next the case $\MUDor = [N_g,\nabla_g,\sigma_g]$ such that there is a connection $\nabla'_g$ with  $\MUDor = [N_g,\nabla_g',0]$ in $\MUDors(X)$. 
Then the rest of the formula can be derived using Lemma \ref{bordismCycleDependenceOnConnection} and the relations in $\MUDors(X)$. 
\end{remark}

\begin{remark} \label{rem:assymetryInPushforward} 
Let $\MUDor=[N_g,\nabla_g,\sigma_g] \in \MUDors(X)$. 
Since $\sigma_g$ is of degree $-1$, we have $d(\sigma_g\wedge h)=d\sigma_g \wedge h - \sigma_g \wedge dh$. 
Hence, modulo $\Imm(d)$, we have $\sigma_g \wedge dh = d\sigma_g \wedge h$, and it follows that modulo $\Imm(d)$ we have
\begin{align}\label{eq:K(nabla_g)_etc}
K(\nabla_g)\wedge h + \sigma_g \wedge (f_*K(\nabla_f) - dh) = (K(\nabla_g)-d\sigma_g) \wedge h + \sigma_g \wedge f_*K(\nabla_f).
\end{align}
Using the maps $\KK$ and $R_{\delta}$ we can rewrite this relation as 
\begin{align*}
K(\nabla_g)\wedge h + \sigma_g \wedge 
R_{\delta}(\geocycle,h)
= \KK(\MUDor) \wedge h + \sigma_g \wedge f_*K(\nabla_f).
\end{align*}
\end{remark}


We will now show that the map $g^\MUDorrep_*$ of Definition \ref{def:pushforward_on_cycles} induces a well-defined pushforward homomorphism on Hodge filtered cobordism. 
We first show that $g^\MUDorrep_*$ sends Hodge filtered bordism data to Hodge filtered bordisms in Lemma \ref{pushforwardIndependentOfCycle}.

\begin{lemma}\label{pushforwardIndependentOfCycle}
We have 
$$
g^\MUDorrep_*\left(BMU_{\delta}^n(p)(X)\right)\subset BMU_{\delta}^{n+2d}(p+d)(Y).
$$
\end{lemma}
\begin{proof}
Let $h\in \widetilde F^{p}\Ds^{n-1}(X;\Vh_*)$. 
By definition of the map $a$ in \eqref{eq:a_forms_def}, 
we have $a(h) = (0,h)$. 
Using relation \eqref{eq:K(nabla_g)_etc} we get
\[
g_*(K(\nabla_g)\wedge h - \sigma_g \wedge dh) = 
g_*((K(\nabla_g)-d\sigma_g)\wedge h).
\]
Since
$g_*((K(\nabla)-d\sigma_g)\wedge h)\in \widetilde F^{p+d}\Ds^{n+2d}(Y;\Vh_*)$, 
we conclude that 
\[
g^\MUDorrep_*(a(h)) \in BMU_{\delta}^{n+2d}(p+d)(Y).
\]
It remains to show 
\[
g^\MUDorrep_*(BMU^n_{\geo}(X))\subset BMU^{n+2d}_{\geo}(Y).
\] 
This follows from \cite[Lemma 4.35]{Bunke2009}. 
We provide a proof for the reader's convenience. 
Let $\widetilde b\in\geocycles^n(\R\times X)$ be a geometric bordism datum on $X$. 
Let $\widetilde e$ denote the geometric cycle $\id_\mathbb R\times g\colon \mathbb R\times X\to \mathbb R\times Y$ with the obvious geometric structure. 
Then $\widetilde e\circ \geocob$ is a geometric bordism datum over $Y$. 

By definition of $\psi(\geocob)$ in \eqref{eq:GeometricBordismDatum}, the fact that $g$ is of even real codimension implies 
\begin{align*}
\psi(\widetilde e\circ \geocob) = g_*(K(\nabla_g)\wedge \psi(\geocob)).
\end{align*}
This shows that we have 
\begin{align*}
\left( \partial(\widetilde e\circ\geocob),\ \psi(\widetilde e\circ\geocob)\right) 
= \left(\widetilde g\circ \partial \geocob,\ g_*\left(K(\nabla_g)\wedge \psi(\geocob)\right)\right). 
\end{align*}
By definition of $g^\MUDorrep_*$, we have 
\begin{align*}
g^\MUDorrep_*(\partial \geocob, \psi(\geocob)) 
= \left(\widetilde g\circ \partial \geocob,\ g_*\left(K(\nabla_g)\wedge \psi(\geocob) +  \sigma_g \wedge R_{\delta}\big(\partial \geocob, \psi(\geocob)\big) \right)\right).
\end{align*}
By \cite[Proposition 2.17]{ghfc}, geometric bordism data satisfy $R\big(\partial \geocob, \psi(\geocob)\big)=0$ and hence also $R_{\delta}\big(\partial \geocob, \psi(\geocob)\big)=0$. 
Thus, we get 
\begin{align*}
g^\MUDorrep_*(\partial \geocob, \psi(\geocob)) 
= \left(\widetilde g\circ \partial \geocob,\ g_*\left(K(\nabla_g)\wedge \psi(\geocob) \right)\right).
\end{align*}
Hence in total we have shown that 
\begin{align*}
g^\MUDorrep_*(\partial \geocob, \psi(\geocob)) 
= \left( \partial(\widetilde e\circ\geocob),\ \psi(\widetilde e\circ\geocob)\right). 
\end{align*}
This shows $g^\MUDorrep_*(BMU^n_{\geo}(X))\subset BMU^{n+2d}_{\geo}(Y)$ and finishes the proof. 
\end{proof}


Next we show that the equivalence class of $g^\MUDorrep_*(\geocycle, h)$ does not depend on the choice of a representative of the $\MUD$-orientation on $g$.  

\begin{lemma}\label{pushforwardIndependentOfRepresentativeOfOrientation}
Let $\MUDorrep=(N_g,\nabla,\sigma)$ and $\MUDorrep'=(N_g, \nabla',\sigma')$ be two representatives of the  $\MUD$-orientation $\MUDor$ of $g \colon X\to Y$. 
Then, for each $\hfcycle\in ZMU_{\delta}^n(p)(X)$, we have
\begin{align*}
[g^\MUDorrep_*\hfcycle]=[g^{\MUDorrep'}_*\hfcycle] ~ \text{in} ~ MU_{\delta}^{n+2d}(p+d)(Y).
\end{align*}
\end{lemma}

\begin{proof} 
Let $\gamma = (\geocycle, h)$ be a currential cycle. 
By the definition of $g^\MUDorrep_*\hfcycle$ we have 
\begin{align*}
[g^\MUDorrep_*(\geocycle, h)] = \left[\widetilde g^\MUDorrep \circ \geocycle, g_*( K(\nabla) \wedge h + \sigma \wedge R_{\delta}(\geocycle,h)) \right].
\end{align*}
Using \eqref{eq:K(nabla_g)_etc} we get
\begin{align}\label{eq:cycle_rep_one}
[g^\MUDorrep_*(\geocycle, h)] = \left[\widetilde g^\MUDorrep \circ \geocycle, g_*((\KK(\epsilon)\wedge h+\sigma\wedge f_*K(\nabla_f)) \right].
\end{align}
Similarly, for the representative $\MUDorrep'$, we get 
\begin{align}\label{eq:cycle_rep_two}
[g^{\MUDorrep'}_*(\geocycle, h)] = \left[\widetilde g^{\MUDorrep'} \circ \geocycle, g_*((\KK(\epsilon')\wedge h+\sigma'\wedge f_*K(\nabla_f)) \right].
\end{align}
We need to show that the two cycles in \eqref{eq:cycle_rep_one} and \eqref{eq:cycle_rep_two}, respectively, are connected by a Hodge filtered bordism. 
By Proposition \ref{prop:K_is_monoid_morphism}, we know 
\begin{align}\label{eq:dsigma-dsigma'=} 
\KK(\epsilon) = K(\nabla)-d\sigma = K(\nabla')-d\sigma' = \KK(\epsilon').
\end{align}
By Remark \ref{rem:equivalent_mud_orientations_over_same_bundle} we can assume $\sigma - \sigma' = \KCS(\nabla,\nabla')$. 
Hence we get
\begin{align*}
\sigma \wedge f_*K(\nabla_f) = \sigma' \wedge f_*K(\nabla_f) + \KCS(\nabla,\nabla') \wedge f_*K(\nabla_f).
\end{align*}
Since $f_*K(\nabla_f)$ is of degree $n$ and $\KCS(\nabla,\nabla')$ is of degree $-1$, switching factor on the right-hand side yields
\begin{align}\label{eq:sigma_sigma'_sign_K_nabla_CSK}
\sigma \wedge f_*K(\nabla_f) = \sigma' \wedge f_*K(\nabla_f) + (-1)^nf_*K(\nabla_f) \wedge \KCS(\nabla,\nabla').
\end{align}
The projection formula $f_*(T\wedge f^*\omega) = (f_*T)\wedge \omega$ applied to the current $T=K(\nabla_f)$ and the form $\omega = \KCS(\nabla,\nabla')$ implies 
\begin{align}\label{eq:K_nabla_CSK_projection_formula}
f_*K(\nabla_f) \wedge \KCS(\nabla,\nabla') 
= f_*(K(\nabla_f) \wedge f^*\KCS(\nabla,\nabla')).
\end{align}
The connections of $\widetilde g^\MUDorrep \circ \geocycle $ and $\widetilde g^{\MUDorrep'}\circ \geocycle $ are $\nabla_f\oplus f^*\nabla$ and   $\nabla_f\oplus f^*\nabla'$, respectively. 
The Chern--Simons form for these two connections satisfies 
\begin{align*}
\KCS(\nabla_f\oplus f^*\nabla,\ \nabla_f\oplus f^*\nabla') = K(\nabla_f)\wedge f^*\KCS(\nabla,\nabla').
\end{align*}
Together with \eqref{eq:K_nabla_CSK_projection_formula} this implies 
\begin{align}\label{eq:representative_CSK_bordism_data_currents}
f_*K(\nabla_f) \wedge \KCS(\nabla,\nabla') = 
f_*\KCS(\nabla_f\oplus f^*\nabla,\ \nabla_f\oplus f^*\nabla').
\end{align}
Hence, identities \eqref{eq:dsigma-dsigma'=}, \eqref{eq:sigma_sigma'_sign_K_nabla_CSK} and \eqref{eq:representative_CSK_bordism_data_currents} together with $g_*\circ f_* = (g \circ f)_*$ on currents show that 
\begin{align}\label{eq:pushforward_rep_currents_become_bordism} 
g_*((\KK(\epsilon)\wedge h & + \sigma\wedge f_*K(\nabla_f)) 
 = g_*((\KK(\epsilon')\wedge h+\sigma'\wedge f_*K(\nabla_f)) \\ 
\nonumber & + (-1)^{n+2d}(g \circ f)_*\KCS(\nabla_f\oplus f^*\nabla,\ \nabla_f\oplus f^*\nabla').
\end{align}
Since $\nabla_f\oplus f^*\nabla$ and   $\nabla_f\oplus f^*\nabla'$ are the connections of $\widetilde g^\MUDorrep \circ \geocycle $ and $\widetilde g^{\MUDorrep'}\circ \geocycle $, respectively, Lemma \ref{bordismCycleDependenceOnConnection} and \eqref{eq:pushforward_rep_currents_become_bordism} imply that the difference of the cycles $g^\MUDorrep_*(\geocycle, h)$ and $g^{\MUDorrep'}_*(\geocycle, h)$ lies in $BMU_{\delta}^{n+2d}(p+d)(Y)$.  
This proves the assertion of the lemma. 
\end{proof}


From now on we will use the canonical isomorphism $\tau \colon MU^n(p)(X) \to MU_{\delta}^n(p)(X)$ of Theorem \ref{equivGeoDesc} to identify $MU^n(p)(X)$ with $MU_{\delta}^n(p)(X)$. 
Putting the previous results together we have shown the following result: 

\begin{theorem}\label{generalPushforwardhfc}
Let $g^\MUDor \colon  X \to Y$ be a proper $\MUD$-oriented holomorphic map with $\MUDor=[\MUDorrep] \in\MUDors(X)$. 
The assignment 
\[
[\geocycle, h] \mapsto 
[g^\MUDorrep_*(\geocycle, h)] 
\]
induces a well-defined homomorphism 
\[
g^\MUDor_* \colon MU^{n}(p)(X) \to MU^{n+2d}(p+d)(Y). 
\]
We refer to $g^\MUDor_*$ as the \emph{pushforward along $g^\MUDor$}.  
\end{theorem}

\begin{remark}
Following Remark \ref{rem:orientation_of_filtration_q} we could consider an orientation $\MUDor_q$ of filtration $q$ for $q\in \Z$. 
Then we would get a pushforward homomorphism 
\[
g^{\MUDor_q}_* \colon MU^{n}(p)(X) \to MU^{n+2d}(p+q+d)(Y)
\]
with an additional shift by $q$. 
Since we are mainly interested in the orientation of Definition \ref{def:canonical_MUD_orientation} which is of filtration $0$ in this terminology, we decided to skip the additional level of generality.  
We note, however, that all the computations in this section could be modified accordingly. 
\end{remark}


The following result shows how the pushforward of Theorem \ref{generalPushforwardhfc} relates to the pushforwards of complex cobordism and sheaf cohomology.

\begin{prop}\label{prop:pushforwards_commute}
Let $g^\MUDor \colon X \to Y$ be a proper $\MUD$-oriented holomorphic map of complex codimension $d$, with $\MUDor = [N_g,\nabla_g,\sigma_g] \in\MUDors(X)$. 
Recall that we write $\KK(\MUDor)=K(\nabla_g)-d\sigma_g$. 
Then the following diagrams commute:
\begin{align}\label{eq:upper_diagram}
\xymatrix{
H^{n-1}\left(X;\frac{\Ds^{*}}{F^p}(\Vh_*)\right)\ar[d]_{g_*\left(\KK(\MUDor)\wedge -\right) }\ar[r]^-a & MU^n(p)(X)\ar[r]^-I\ar[d]^{g^\MUDor_*}& MU^n(X)\ar[d]^{g_*}  \\
H^{n-1+2d}\left(Y;\frac{\Ds^{*}}{F^{p+d}}(\Vh_*)\right)  \ar[r]^-a & MU^{n+2d}(p+d)(Y)\ar[r]^-I & MU^{n+2d}(Y)
}
\end{align}
\begin{align}\label{eq:lower_diagram}
\xymatrix{
MU^n(p)(X)\ar[r]^-R \ar[d]_{g^\MUDor_*}& H^n(X;F^p\Ds^*(X;\Vh_*))\ar[d]^{g_*\left(\KK(\MUDor)\wedge -\right)}\\
MU^{n+2d}(p+d)(Y)\ar[r]^-R& H^{n+2d}(X;F^{p+d}\Ds^*(Y;\Vh_*)).
}
\end{align}
\end{prop}

\begin{proof}
For $[h]\in H^{n-1}\left(X;\frac{\Ds^{*}}{F^p}(\Vh_*)\right)$ we have
\begin{align*}
    g^\MUDor_*(a[h]) & = [0, g_*(\KK(\MUDor)\wedge h)]\\
        & = a\left(g_*(\KK(\MUDor)\wedge h)\right)
\end{align*}
which proves that the left-hand square in \eqref{eq:upper_diagram} commutes. 
That the right-hand square in \eqref{eq:upper_diagram} commutes follows 
from the observation that the underlying complex-oriented map of a composition of geometric cycles, is the composition of the underlying complex-oriented maps. 
Hence we have
\[
g_*(I[\geocycle, h])= g_*[f,N_f] = [g\circ f, N_f\oplus f^*N_g] = I(g^{\MUDor}_*[\geocycle, h]).
\]
Finally, by Lemma \ref{pushforward_current_wedge} 
we have 
\begin{align*}
R( g^\MUDor_*(\hfcycle)) & = g_*\left(\KK(\MUDor) \wedge R(\hfcycle)\right)
\end{align*}
which shows that square \eqref{eq:lower_diagram} commutes as well.
\end{proof}


We will now show that the pushforward is functorial:

\begin{theorem}\label{thm:pushforward_is_functorial}
Let the composition of proper holomorphic maps 
\[
g_2\circ g_1 \colon X_1 \xto{g_1} X_2 \xto{g_2} X_3
\]
be endowed with the composed $\MUD$-orientation $\MUDor_1 + g_1^*\MUDor_2$. 
Then we have 
\[
(g_{2*}^{\MUDor_2}) \circ (g_{1*}^{\MUDor_1}) = (g_2\circ g_1)^{\MUDor_1 + g_1^*\MUDor_2}_*
\]
as homomorphisms 
\[
MU^n(p)(X_1)\to MU^{n+2d_1+2d_2}(p+d_1+d_2)(X_3).
\]
\end{theorem}
\begin{proof} 
Let $\hfcycle = (\geocycle, h) \in ZMU_{\delta}^n(p)(X_1)$. 
For $i=1,2$, let $\epsilon_i = (N_{i},\nabla_{i},\sigma_i)$ represent $\MUDor_i$. 
We use the representative of $\MUDor_1+g_1^*\MUDor_2$ suggested by Lemma
\ref{lem:formula_for_addition_in_mudors},
\[
\MUDorrep_{12} = (N_{1} \oplus g_1^*N_{2},\nabla_{1} \oplus g_1^*\nabla_{2}, \sigma_{12})
\]
with 
\begin{align}\label{eq:formula_for_sigma12}
\sigma_{12} := \sigma_1 \wedge g_1^*\KK(\MUDor_2) +  g_1^* \sigma_2 \wedge K(\nabla_{1}).
\end{align}
Observe that the underlying geometric cycle of $g_{12}^{\epsilon_{12}}$, which we denote by $\widetilde g_{12}$, is the composed geometric cycle  
$\widetilde g_{12}=\widetilde g_2\circ \widetilde g_1$.
Therefore, we know that 
the underlying geometric cycles of 
\begin{align*}
(g_{12}^{\epsilon_{12}})_*\hfcycle &= (\widetilde g_{12}\circ \geocycle,\ h_{12}) \quad \text{and}\\
(g_2^{\epsilon_2})_*\circ (g_1^{\epsilon_1})_*\hfcycle &= (\widetilde g_2\circ \widetilde g_1\circ \geocycle ,\ h_{\circ})
\end{align*}
coincide. 
It remains to show that $h_{12}=h_{\circ}$ modulo $\Imm(d)$.  
By definition of the pushforward and Remark \ref{rem:assymetryInPushforward} we have 
\begin{align*}
h_{12} &=(g_2\circ g_1)_*\Big[\KK(\MUDor_1 + g^*_1\MUDor_2)\wedge h + \sigma_{12} \wedge f_*K(\nabla_f)\Big].
\end{align*}
On the other hand, applying Remark \ref{rem:assymetryInPushforward} to $h_{\circ}$ yields 
\begin{align*}
h_{\circ} = g_{2*} \Big[ \KK(\MUDor_2) \wedge g_{1*} \big( \KK(\MUDor_1)\wedge h  + \sigma_1\wedge f_* K(\nabla_f) \big) 
+ \sigma_2\wedge (g_1\circ f)_*K(\nabla_{g_1\circ f})  \Big].
\end{align*}
Now we use the projection formula $f_*(T\wedge f^*\omega) = (f_*T)\wedge \omega$ for a current $T$ and a form $\omega$.  
Since we have $K(\nabla_{g_1\circ f}) = K(\nabla_f) \wedge f^*K(\nabla_{1})$, the projection formula for $T=K(\nabla_f)$ and $\omega = K(\nabla_1)$ implies 
\begin{align*}
f_*K(\nabla_{g_1\circ f}) = f_*(K(\nabla_f)\wedge f^*K(\nabla_{1})) = f_*K(\nabla_f)\wedge K(\nabla_1).
\end{align*}
Hence we can rewrite $h_\circ$ as 
\begin{align*}
h_{\circ} = g_{2*} \Big[ \KK(\MUDor_2) \wedge g_{1*} \big( \KK(\MUDor_1)\wedge h  + \sigma_1\wedge f_* K(\nabla_f) \big) 
+ \sigma_2 \wedge g_{1*} (f_*K(\nabla_f) \wedge K(\nabla_1))  \Big].
\end{align*}
We apply again the projection formula to the pushforward along $g_1$, once with $T = \KK(\MUDor_1)\wedge h  + \sigma_1\wedge f_* K(\nabla_f)$ and $\omega = \KK(\MUDor_2)$, 
and once with $T = f_*K(\nabla_f) \wedge K(\nabla_1)$ and $\omega = \sigma_2$. 
Since $g_1^*\KK(\MUDor_2)$ and $K(\nabla_1)$ lie in $\Ah^0(X_1;\Vh_*)$, and hence in the center of the ring $\Ah^*(X_1;\Vh_*)$, we then get 
\begin{align*}
h_\circ = (g_2\circ g_1)_*\Big[g_1^*\KK(\MUDor_2)\wedge\big(\KK(\MUDor_1)\wedge h +\sigma_1\wedge f_*K(\nabla_f) \big) + g_1^*\sigma_2\wedge K(\nabla_1)\wedge f_* K(\nabla_f) \Big].
\end{align*}
Next we collect the terms that are wedged with $f_*K(\nabla_f)$ and obtain:
\begin{align*}
h_\circ=(g_2\circ g_1)_*\Big[g_1^*\KK(\MUDor_2)\wedge\KK(\MUDor_1)\wedge h + \big(g_1^*\KK(\MUDor_2)\wedge\sigma_1 + g_1^*\sigma_2\wedge K(\nabla_1)\big) \wedge f_*K(\nabla_f)\Big].
\end{align*}
By Proposition \ref{prop:K_is_monoid_morphism} we have $\KK(\MUDor_1 + g_1^*\MUDor_2) = \KK(\MUDor_1) \wedge \KK(g_1^*\MUDor_2)$ which implies 
\begin{align*}
h_\circ=(g_2\circ g_1)_*\Big[ \KK(\MUDor_1 + g_1^*\MUDor_2) \wedge h + \big(g_1^*\KK(\MUDor_2)\wedge\sigma_1 + g_1^*\sigma_2\wedge K(\nabla_1)\big) \wedge f_*K(\nabla_f)\Big].
\end{align*}
Finally, by formula \eqref{eq:formula_for_sigma12} for $\sigma_{12}$, we get 
\begin{align*}
h_{\circ} =(g_2\circ g_1)_*\Big[\KK(\MUDor_1 + g^*_1\MUDor_2)\wedge h + \sigma_{12} \wedge f_*K(\nabla_f)\Big].
\end{align*}
This shows $h_\circ=h_{12}$ and finishes the proof. 
\end{proof}

\begin{remark}\label{rem:product_of_pushforwards_and_pi}
Let $g \colon X \to Y$ and $q \colon W \to Y$ be transverse 
proper holomorphic maps of codimensions $d$ and $d'$, respectively. 
Let $\pi \colon W \times_Y X \to Y$ be the map induced by the following cartesian diagram in $\Manc$
\begin{align*}
\xymatrix{
W \times_Y X \ar[d]_-{g'} \ar[r]^-{q'} \ar@{.>}[dr]^-{\pi} & X \ar[d]^-g \\
W \ar[r]_-q & Y.
}
\end{align*}
Let $\MUDor_g \in \MUDors(X)$ and $\MUDor_q \in \MUDors(W)$ be $\MUD$-orientations of $g$ and $q$, respectively. 
We then have natural isomorphisms of stable normal bundles  
$(g')^*N_q = N_{q'}$, $(q')^*N_g = N_{g'}$, and  
\begin{align*}
N_{\pi} = (g')^*N_q \oplus N_{g'} = (q')^*N_g \oplus N_{q'}. 
\end{align*}
Hence $(g')^*\MUDor_q + (q')^*\MUDor_g$ is an orientation of $\pi$, and Theorem \ref{thm:pushforward_is_functorial} implies that we have the following identity 
\begin{align*}
g_*^{\MUDor_g} \circ (q')_*^{(g')^*\MUDor_q} = \pi_*^{(g')^*\MUDor_q + (q')^*\MUDor_g} 
= q_*^{\MUDor_q} \circ (g')_*^{(q')^*\MUDor_g} 
\end{align*}
of homomorphisms $MU^{n}(p)(W \times_Y X) \to MU^{n+2d+2d'}(p+d+d')(Y)$. 
\end{remark}


Next we will show that the pushforward is compatible with pullbacks. 
First we briefly recall the construction of  pullback homomorphisms in $MU^*(p)(-)$ from \cite[Theorem 2.22]{ghfc}. 
For further details we refer to \cite[\S 2.7]{ghfc} and the references therein.  
Let $k \colon Y' \to Y$ be a holomorphic map.  
We consider the following cartesian diagram of manifolds
\[
\xymatrix{Z'\ar[r]^{k_Z}\ar[d]_{k^*f} & Z\ar[d]^f\\ Y'\ar[r]_k & Y}
\]
where $k$ and $f$ are transverse, and $\geocycle = (f,N_f,\nabla_f)$ is a geometric cycle on $Y$.
By transversality we get that $k^*f$ is complex-oriented with $N_{k_Z}={k_Z}^*N_f$. 
We define $k^*\geocycle$ by
\[
k^*\geocycle = (k^*f, k_Z^*N_f, k_Z^*\nabla_f).
\]
For a cycle $(\geocycle,h) \in ZMU^n(p)(Y)$, it remains to define the pullback of the current $h$. 
Since the pullback of an arbitrary current is not defined, this requires to restrict to the subgroup $ZMU_k^n(p)(Y) \subset ZMU^n(p)(Y)$ consisting of those $\gamma=(\geocycle,h)$ satisfying 
\begin{align*}
\mathrm{WF}(h)\cap N(k)=\emptyset ~ \text{and} ~ 
k \pitchfork f 
\end{align*}
where $\mathrm{WF}(h)$ denotes the wave-front set of $h$ and $N(k)$ is the normal set of $f$ as defined in \cite[8.1]{Hoer}. 
For $\gamma = (\geocycle,h) \in ZMU_k^n(p)(Y)$, we then have a well-defined pullback 
\begin{align*}
k^*\gamma = k^*(\geocycle,h) = (k^*f, k_Z^*N_f, k_Z^*\nabla_f,k^*h)
\end{align*}
where $k^*h$ is well-defined by \cite[Theorem 8.2.4]{Hoer}. 
By \cite[Theorem 2.25]{ghfc}, this induces a pullback homomorphism 
\begin{align*}
k^*\colon MU^n(p)(Y)\to MU^n(p)(Y'). 
\end{align*}

\begin{theorem}\label{thm:pushpull}
Suppose we have a cartesian diagram in $\Manc$ 
\begin{align}\label{eq:pullback_diagram_pushforward}
\xymatrix{
X' \ar[d]_-{g'} \ar[r]^-{k'} & X \ar[d]^-g \\
Y' \ar[r]_-k & Y
}
\end{align}
with $k$ transverse to $g$, and $g$ proper of codimension $d$. 
Let $\MUDor$ be an $\MUD$-orientation of $g$. 
We equip $g'$ with the pullback orientation $\MUDor':={k'}^*\MUDor$. 
Then we have 
\[
k^*g^{\MUDor}_* = (g')^{\MUDor'}_*{k'}^* \colon MU^n(p)(X) \to MU^{n+2d}(p+d)(Y'). 
\]
\end{theorem}
\begin{proof}
Let $\hfcycle=(\geocycle,h) = (f,N_f,\nabla_f,h) \in ZMU^n(p)(X)$ be a cycle. Since transversality is generic we can assume $f$ to be transverse with $k'$.
Let $k'_Z \colon Z' \to Z$ be the induced map in the top cartesian rectangle in 
\begin{align*}
\xymatrix{
Z' \ar[d]_-{{k'}^*f} \ar[rr]^-{k'_Z = k_Z} & & Z \ar[d]^-f \\
X' \ar[d]_-{g'} \ar[rr]^-{k'} & & X \ar[d]^-g \\
Y' \ar[rr]_-k & & Y.
}
\end{align*}
Since both rectangles are cartesian, the outer rectangle is cartesian as well. 
Hence the map $k_Z \colon Z' \to Z$ induced by the outer cartesian diagram agrees with $k'_Z$. 
We write 
\[
k^*g^{\MUDor}_*(\geocycle,h) =: (f_{\lrcorner},h_{\lrcorner}) ~ 
\text{and} ~ 
(g')^{\MUDor'}_*{k'}^*(\geocycle,h) =: (f^{\ulcorner},h^{\ulcorner}). 
\]
Let $\MUDorrep = (N_g,\nabla_g,\sigma_g)$ be a representative of the orientation $\MUDor$ of $g$. 
Then we have 
\begin{align*}
f_{\lrcorner} & = (k^*(g \circ f), k_Z^*N_f \oplus k_Z^*f^*N_g, k_Z^*\nabla_f \oplus k_Z^*f^*\nabla_g).
\end{align*}
The pullback orientation ${k'}^*\MUDor$ is represented by $\MUDorrep' = ({k'}^*N_g,{k'}^*\nabla_g,{k'}^*\sigma_g)$. 
Since diagram \eqref{eq:pullback_diagram_pushforward}, 
is cartesian
we have
\begin{align*}
f_{\lrcorner}
& = ((g'\circ {k'}^*f), N_{{k'}^*f} \oplus ({k'}^*f)^*N_g, \nabla_{{k'}^*f} \oplus ({k'}^*f)^*\nabla_g) \\ 
& = f^{\ulcorner}. 
\end{align*}
Now we check the effect on the current $h$ using that we have $k^*g_* = g'_*{k'}^*$ by \cite[Theorem 2.27]{ghfc} whenever the involved maps are defined: 
\begin{align*}
h_{\lrcorner} & = k^*(g_*\big( K(\nabla_ g)\wedge h + \sigma_g\wedge (f_*K(\nabla_f)-dh) \big) \\
 & = g'_*\big( {k'}^*K(\nabla_ g) \wedge {k'}^*h + {k'}^*\sigma_g \wedge {k'}^*(f_*K(\nabla_f)-dh) \big) \\
 & = h^{\ulcorner}. \qedhere
\end{align*}
\end{proof}

We recall from \cite[\S 2.8]{ghfc} that there is a natural product of the form 
\begin{align}\label{intprod}
    MU^{n_1}(p_1)(X)\times MU^{n_2}(p_2)(X)& \to MU^{n_1+n_2}(p_1+p_2)(X) 
\end{align} 
turning $MU^*(*)(X)$ into a ring. 
The product of two classes $[\gamma_1]$ and $[\gamma_2]$ is denoted by $[\gamma_1]\cdot [\gamma_2]$ and is induced by the following construction: 
We consider the operation
\begin{align}\label{eq:otimes_product}
\otimes \colon \Ds^{n_1}(X_1;\Vh_*)\times \Ds^{n_2}(X_2;\Vh_*)\to \Ds^{n_1+n_2}(X_1\times X_2;\Vh_*) 
\end{align}
satisfying $T_1\otimes T_2=\pi_1^*T_1\wedge \pi_2^*T_2$. 
Since $K$ is multiplicative, we have 
\[
K^{p_1+p_2}(\nabla_1\oplus \nabla_2)=K^{p_1}(\nabla_1)\otimes K^{p_2} (\nabla_2).
\]
We then define the symbol $\times$, and refer to it as the external product of Hodge filtered cycles by potential slight abuse of terminology, by
\begin{align}\label{eq:exteriorProductHFcycleDef}
\hfcycle_1\times \hfcycle_2 := 
\left( \widetilde f_1\times\widetilde f_2, 
h_1\otimes R(\gamma_2) + (-1)^{n_1} {f_1}_*K(\nabla_{f_1})\otimes h_2\right). 
\end{align}

The product in \eqref{intprod} is then defined as the pullback along the diagonal map $\Delta_X \colon X \to X \times X$: 
\[
[\hfcycle_1]\cdot [\hfcycle_2] = \Delta_X^*([  \hfcycle_1\times \hfcycle_2]).
\]
The following theorem shows that $g_*$ is a homomorphism of $MU^*(*)(Y)$-modules. 


\begin{theorem}\label{thm:projection_formula}
Let $g \colon X \to Y$ be a proper holomorphic map of codimension $d$ and let $\MUDor$ be an $\MUD$-orientation of $g$. 
Then, for all integers $n$, $p$, $m$, $q$, and all elements $x\in MU^n(p)(X)$ and $y\in MU^m(q)(Y)$, we have the following projection formula
\begin{align*}
g_*^{\MUDor}\left( g^*y \cdot x \right) = y \cdot g_*^{\MUDor}x ~ \text{in} ~ MU^{n+m+2d}(p+q+d)(Y).
\end{align*}
\end{theorem}
\begin{proof}
Since the product is defined by pulling back an exterior product along the diagonal, 
we consider the following commutative diagram 
\begin{align}\label{eq:diagram_exterior_G_and_g}
\xymatrix{
X \ar[d]_-g \ar[rr]^-{(g, \id_X)} & & Y \times X \ar[d]_-{G = \id_Y\times g} \ar[r]^-{{\pi_X}} & X \ar[d]^-g\\
Y \ar[rr]_-{\Delta_Y} & & Y \times Y \ar[r]_-{\pr_2} & Y.
}    
\end{align}

We denote by $\pi_Y \colon Y \times X \to Y$ and $\pi_X \colon Y \times X \to X$, and by $\pr_1 \colon Y \times Y \to Y$ and $\pr_2 \colon Y \times Y \to Y$  the projections onto the first and second factors, respectively.  
We endow the map $G:=\id_Y \times g$ with the pullback $\MUD$-orientation $\MUDor':=\pi_X^*\MUDor$. 
We claim that in order to prove the assertion of the theorem it suffices to show the identity 
\begin{align}\label{eq:external_projection_formula_classes}
G^{\MUDor'}_*(y \times x) = y \times g^{\MUDor}_*(x).
\end{align}
To prove that it suffices to show \eqref{eq:external_projection_formula_classes}, we observe that \eqref{eq:external_projection_formula_classes} implies that 
\begin{align*}
\Delta_Y^*G^{\MUDor'}_*(y \times x) = \Delta_Y^*(y \times g^{\MUDor}_*(x)) = y \cdot g_*^{\MUDor}x
\end{align*}
by definition of the cup product on $MU^*(*)(Y)$. 
Hence it remains to show that 
\begin{align*}
\Delta_Y^*G^{\MUDor'}_*(y \times x) = g_*^{\MUDor}\left( g^*y \cdot x \right). 
\end{align*}
To do so we consider the following diagram 
\begin{align}\label{eq:diagram_diagonals_and_G_and_g}
\xymatrix{
X \ar[d]_-g \ar[rr]^-{\Delta_X} & & X \times X \ar[d]_-{g\times g} \ar[r]^-{g \times \id_X} & Y \times X \ar[d]^-G\\
Y \ar[rr]_-{\Delta_Y} & & Y \times Y \ar[r]_-{\id_{Y \times Y}} & Y \times Y.
}    
\end{align}
Since the outer diagram in \eqref{eq:diagram_diagonals_and_G_and_g} is cartesian and since $G$ and $\Delta_Y$ are transverse, 
we can apply Theorem \ref{thm:pushpull} to get 
\begin{align*}
\Delta_Y^*G^{\MUDor'}_*(y \times x) 
& = g_*^{\MUDor}((g \times \id_X) \circ \Delta_X)^*(y \times x) \\
& = g_*^{\MUDor}(\Delta_X^*((g \times \id_X)^*(y \times x)) \\
& = g_*^{\MUDor}(\Delta_X^*((g^*y \times x)) \\
& = g_*^{\MUDor}\left( g^*y \cdot x \right)
\end{align*}
where the last equality uses the definition of the cup product on $MU^*(*)(X)$. 
This proves the claim. 

We will now show that identity \eqref{eq:external_projection_formula_classes} holds by proving the corresponding formula on the level of cycles. 
Let $\MUDorrep = (N_g,\nabla_g,\sigma_g)$ be a representative of $\MUDor$. 
Then $\MUDorrep' = (\pi_X^*N_g,\pi_X^*\nabla_g,\pi_X^*\sigma_g)$ represents $\MUDor' = \pi_X^*\MUDor$. 
Let $(\geocycle_x,h_x)$ and $(\geocycle_y,h_y)$ be cycles such that $x = [\geocycle_x,h_x]$ and $y = [\geocycle_y,h_y]$. 
We write $h_{y\times x}$ for the current defined by \eqref{eq:exteriorProductHFcycleDef} such that $y \times x = [\geocycle_y \times \geocycle_x, h_{y\times x}]$. 
The theorem will then follow once we have proven the identity of cycles 
\begin{align}\label{eq:external_projection_formula}
G^{\MUDorrep'}_*(\geocycle_y \times \geocycle_x, h_{y\times x}) = (\geocycle_y, h_y) \times  g^{\MUDorrep}_*(\geocycle_x, h_x).
\end{align}

Formula \eqref{eq:external_projection_formula} can be checked separately on the level of geometric cycles and on the level of currents. 
To simplify the notation we denote the cycle   $G^{\MUDorrep'}_*(\geocycle_y \times \geocycle_x, h_{y\times x})$ by $(\geocycle_G, h_G)$. 
We write $(\geocycle_{g_*(x)}, h_{g_*(x)})$ for the cycle $g^{\MUDorrep}_*(\geocycle_x , h_x)$, 
and $(\geocycle_{y\times g_*(x)}, h_{y \times g_*(x)})$ for the cycle $(\geocycle_y,h_y) \times g^{\MUDorrep}_*(\geocycle_x , h_x)$.  
For the geometric cycles the formula $\geocycle_G = \geocycle_{y \times g_*(x)}$ follows directly from the definition of the pushforward and the definition of the map $G=\id_Y \times g$.

Now we show that \eqref{eq:external_projection_formula} holds for the corresponding currents. 
Recall that we use the notation $\phi(\gamma) = (f_{\gamma})_*K(\nabla_{f_{\gamma}})$ and $R(\gamma) = \phi(\gamma) - dh_{\gamma}$ for a cycle $\gamma=(f_{\gamma},h_{\gamma})$.  
We then have by definition of the exterior product $\times$ 
\begin{align}\label{eq:h_y_times_h_x}
h_{y \times x} = h_y\otimes R(x) + (-1)^{m} \phi(y) \otimes h_x. 
\end{align}
By definition of the pushforward we have 
\begin{align*}
h_G =  G_*\Big( \pi_X^*K(\nabla_g) \wedge h_{y \times x}  + \pi_X^*\sigma_g \wedge R(y \times x) \Big). 
\end{align*}
Using formula \eqref{eq:h_y_times_h_x} and the formula $R(y\times x) = R(y)\otimes R(x)$, which is verified in \cite[page 26]{ghfc}, 
we can rewrite $h_G$ as
\begin{align*}
h_G  & =  G_*\Big( \pi_X^*K(\nabla_g) \wedge (h_y\otimes R(x) + (-1)^{m} \phi(y) \otimes h_x) + \pi_X^*\sigma_g \wedge (R(y) \otimes R(x)) \Big). 
\end{align*}
By definition of $\otimes$ in \eqref{eq:otimes_product} and the fact that $R(y)$ is of degree $m$ we then get 
\begin{align*}
h_G = G_*\Big( h_y \otimes (K(\nabla_g) \wedge R(x)) & + (-1)^{m} \phi(y) \otimes (K(\nabla_g) \wedge h_x) \\
& + (-1)^m R(y) \otimes (\sigma_g \wedge R(x)) \Big).
\end{align*}
Now we use the definition of $G$ as $G=\id_Y\times g$ to get:
%
\begin{align*}
h_G =  h_y \otimes g_*(K(\nabla_g) \wedge R(x)) & + (-1)^m \phi(y) \otimes g_*(K(\nabla_g) \wedge h_x) \\
& + (-1)^m R(y) \otimes g_*(\sigma_g \wedge R(x)). 
\end{align*}
On the other hand we compute 
\begin{align*}
 & ~ h_{y \times g^\MUDorrep_*(x)} \\
= & ~ h_y\otimes R(g^{\MUDorrep}_*(\geocycle_x,h_x)) 
+ (-1)^m\phi(y)\otimes h_{g^\MUDorrep_*(x)} \\
= & ~ h_y\otimes g_*\big( (K(\nabla_g) -d\sigma_g)\wedge R(x) \big) +(-1)^m \phi(y) \otimes g_* \big( K(\nabla_g) \wedge h_x +\sigma_g \wedge R(x) \big) \\
= & ~ h_y\otimes g_*(K(\nabla_g) \wedge R(x)) - h_y \otimes g_*(d\sigma_g \wedge R(x)) \\
 & ~ + (-1)^m \phi(y) \otimes g_* (K(\nabla_g) \wedge h_x) + (-1)^m \phi(y) \otimes g_*(\sigma_g \wedge R(x)). 
\end{align*}
Comparing the expressions for 
$h_G$ and $h_{y \times g^\MUDorrep_*(x)}$ it remains to show 
\begin{align*}
(-1)^m R(y) \otimes g_*(\sigma_g \wedge R(x)) 
+  h_y \otimes g_*(d\sigma_g \wedge R(x)) 
= (-1)^m \phi(y) \otimes g_*(\sigma_g \wedge R(x))
\end{align*} 
modulo $\Imm(d)$. 
Since, by definition of $R$ in \eqref{cyclstructmaps}, $R(x)$ is a closed form, we have 
\begin{align*}
d(\sigma_g \wedge R(x)) = d\sigma_g \wedge R(x) - \sigma_g \wedge dR(x) = d\sigma_g \wedge R(x).
\end{align*}
Since $h_y$ 
is of degree $m$, we therefore get 
\begin{align*}
d(h_y \otimes g_*(\sigma_g \wedge R(x))) 
& = dh_y \otimes g_*(d\sigma_g \wedge R(x)) 
+ (-1)^m h_y \otimes g_*(d\sigma_g \wedge R(x)). 
\end{align*} 
Hence, modulo image of $d$, we get the following identity 
\begin{align*}
h_y \otimes g_*(d\sigma_g \wedge R(x)) 
= (-1)^m dh_y \otimes g_*(d\sigma_g \wedge R(x)) \quad \text{modulo} ~ \Imm(d).
\end{align*} 
Since $R(y) = \phi(y)-dh_y$ by definition, we can thus conclude  
\begin{align*}
& ~ (-1)^m R(y) \otimes g_*(\sigma_g \wedge R(x)) 
+ h_y \otimes g_*(d\sigma_g \wedge R(x)) \\
= & ~ (-1)^m R(y) \otimes g_*(\sigma_g \wedge R(x)) 
+ (-1)^m dh_y \otimes g_*(\sigma_g \wedge R(x)) \\
= & ~ (-1)^m \phi(y) \otimes g_*(\sigma_g \wedge R(x))
\end{align*} 
modulo $\Imm(d)$. 
This shows \eqref{eq:external_projection_formula} and finishes the proof. 
\end{proof}


We end this section with a further observation on the relationship of the maps $R$ and $\KK$.

\begin{remark}\label{rem:current_of_pushforward_formula}
As in the proof of Proposition \ref{prop:pushforwards_commute}, we can  express the identity shown in Lemma \ref{pushforward_current_wedge} as
\begin{align*}
R(g^\MUDor_*(\hfcycle)) & = g_*(\KK(\MUDor) \wedge R(\hfcycle))
\end{align*}
for every element $[\gamma]$ and proper holomorphic map $g \colon X \to Y$ with $\MUD$-orientation $\MUDor$.  
For the special case that $\gamma$ is the identity element $1_X$ of the ring $MU^*(*)(X)$, i.e., for
$[\gamma] = 1_X = [\id_X,d,0]$, 
we get 
\begin{align*}
R(g^\MUDor_*(1_X)) & = g_*(\KK(\MUDor)).
\end{align*}
\end{remark}


\section{A canonical Hodge filtered \texorpdfstring{$MU$}{MU}-orientation for holomorphic maps}\label{sec:canonical_MUD_orientation}

We will now show that for every holomorphic map there is a natural choice for an $\MUD$-orientation. 
The key result is Theorem \ref{thm:Bott_connection_gives_an_orientation} which provides us with a canonical choice of a class of connections. 
The existence of a canonical choice of a class of orientation and Theorem \ref{thm:canonical_pushforward} 
may be seen as justification for defining $\MUD$-orientations as a $K$-group and not just as a set. 
We recall from \cite[\S 6.3]{karoubi43} the following terminology.

\begin{defn}
Let $X$ be a complex manifold and let $D$ be a smooth connection on a holomorphic vector bundle $E$ over $X$. 
Then with respect to local coordinates $(U_i,g_i)$, $D$ acts as $d+\Gamma_i$, where $\Gamma_i=(\Gamma_i^{jk})$ is a matrix of $1$-forms. 
Recall that we have
\[
\Gamma_i=g_{ji}^{-1}dg_{ji} + g_{ji}^{-1}\Gamma_jg_{ji}
\]
where the $g_{ij}$ denote the transition functions. 
Conversely any such cocycle $\{\Gamma_i\}$ defines a connection.  
Then $D$ is called a \emph{Bott connection} if for each $i,j,k$ we have
\[
\Gamma_i^{jk}\in F^1\Ah^1(X).
\] 
\end{defn}

\begin{remark}
As noted in the introduction, Bott connections are more commonly referred to as connections compatible with the holomorphic structure. 
Here we follow Karoubi, who uses the terminology in \cite{karoubi43} in a context where Bott connections generalize both connections compatible with a holomorphic structure and Bott connections of foliation theory. 
Since Bott connections are frequently used in what follows, we adopt \emph{Bott} connection as a convenient terminology.
\end{remark}

\begin{remark}\label{rem:Bott_and_Chern_connections}
Every holomorphic vector bundle on a complex manifold admits a Bott connection. 
In fact, the Chern connection on a holomorphic bundle with a hermitian metric is defined as the unique Bott connection which is compatible with the hermitian structure. 
By \cite[Proposition 4.1.4]{huybrechts} every complex vector bundle admits a hermitian metric. 
By \cite[Proposition 4.2.14]{huybrechts} every holomorphic bundle with a hermitian structure has a Chern connection. 
Alternatively, one can show the existence of Bott connections as in \cite[\S 6]{karoubi43} using a local trivialization of the bundle and a partition of unity.\footnote{We emphasize again that a Bott connection does not have to be holomorphic, but is merely required to be smooth. Hence one may use a partition of unity for the construction as explained in \cite[\S 6]{karoubi43}.}
\end{remark}

\begin{remark}\label{rem:K_of_Bott_connection_is_in_F0A0}
If $D$ is a Bott connection, then the curvature of $
D$, which in local coordinates is represented by the matrix
\[
d\Gamma_i + \Gamma_i\wedge \Gamma_i,
\]
belongs to $F^1\Ah^2( X;\mathrm{End}(E))$. 
This implies the following key fact about Bott connections: 
\begin{align}\label{eq:K_of_Bott_connection_is_in_F0A0}
K(D)\in F^0\Ah^0(X;\Vh_*).
\end{align}
\end{remark}

\begin{remark}
Let $D$ be a Bott connection on $E$. 
Then \eqref{eq:K_of_Bott_connection_is_in_F0A0} implies that the triple $(E,D,0)$ defines an element in $\MUDors(X)$. 
The following result, inspired by \cite[Theorem 6.7]{karoubi43}, shows that the associated orientation class $[E,D,0]$ is independent of the choice of Bott connection $D$. 
\end{remark}

We will now prove the key technical result of this section. 

\begin{theorem}\label{thm:Bott_connection_gives_an_orientation}
For every $X\in\Manc$, there is a natural homomorphism 
\[
B \colon K_{\hol}^0(X)\to \MUDors(X)
\]
induced by  
\begin{align*}
B[E] = [E,D,0],
\end{align*}
for each holomorphic vector bundle $E$ where $D$ is any Bott connection on $E$. 
\end{theorem}
\begin{proof}
The existence of a Bott connection was pointed out in Remark \ref{rem:Bott_and_Chern_connections}. 
The assertion of the theorem then follows from the following two lemmas. 
\end{proof}

As a first step we analyze the Chern--Simons form of two Bott connections on a given holomorphic vector bundle and show that they lead to the same orientation class: 

\begin{lemma} \label{lemma:Bott_connections_are_homotopic}
Let $D$ and $D'$ be two Bott connections for a holomorphic vector bundle $E\to X$. Then $\KCS(D,D')\in F^0\Ah^0(X;\Vh_*)$, so $[E,D,0]=[E,D',0]$ in $\MUDors(X)$.
\end{lemma}
\begin{proof}
Let $\Gamma_i$ and $\Gamma'_i$ be the connection matrices of $D$ and $D'$, respectively, with respect to local holomorphic coordinates $z_1,\dots, z_l$ on $U_i$.
Let $I=[0,1]$ be the unit interval and let $\pi \colon I \times X \to X$ denote the projection. 
Then consider the connection $D''=t\cdot \pi^*D+(1-t)\cdot\pi^*D'$ on $\pi^*E$. Its connection matrix on $I\times U_i$ is 
\[
\Gamma'' = t\Gamma_i + (1-t)\Gamma'_i.
\]
The curvature of $D''$ is given on $I\times U_i$ by 
\begin{align*}
\Omega''_i =& d\Gamma''_i + \Gamma''_i\wedge \Gamma''_i \\
=& dt\wedge \Gamma_i + t\cdot d\Gamma_i -dt \wedge\Gamma'_i + (1-t)d\Gamma'_i + t^2\Gamma_i\wedge \Gamma_i \\
&+ (1-t)^2\Gamma'_i\wedge \Gamma'_i + t(1-t)\Gamma_i\wedge\Gamma'_i. 
\end{align*}
Each term is of filtration $1$ in the sense that at least one of the $dz_j$s appears in each term of each entry.
Hence the Chern form $c_k(D'')$ has at least $k$ many $dz_j$s appearing in its local expression, and in that sense belongs to $F^k\Ah^*([0,1]\times X)$. 
Integrating out $dt$ maps this filtration step  $F^k\Ah^*([0,1]\times X)$ to the Hodge filtration $F^k\Ah^*(X)$. 
This implies 
\begin{align*} 
\pi_*K(D'')\in F^0\Ah^{-1}(X;\Vh_*).
\end{align*}
Since $\KCS(D,D')= \pi_* K(D'')$, this proves 
\[
[E,D,0] = [E,D',0] \in \MUDors(X). \qedhere
\]
\end{proof}

Next, we show that all the defining relations of $K_{\hol}^0(X)$ and $\MUDors(X)$ are respected by $B$: 

\begin{lemma}
Let $E$ be a holomorphic bundle over $X$ and $D$ be a Bott connection on $E$. 
The assignment $E\mapsto (E,D,0)$ induces a map $B\colon K_{\hol}^0(X) \to \MUDors(X)$.
\end{lemma}
\begin{proof}
The first part of this proof follows \cite[Proof of Theorem 6.7]{karoubi43}. 
Let 
\begin{align*}
\xymatrix{0\ar[r] & E_1\ar[r]^\alpha & E_2\ar[r]^\beta & E_3\ar[r] & 0}
\end{align*}
be a short exact sequence of holomorphic vector bundles, 
and let $D_i$ be a Bott connection on $E_i$. 
By the defining relations for $\MUDors(X)$ we need to establish 
\begin{equation*}
\KCS(D_1,D_2,D_3) \in \widetilde{F}^0\Ah^{-1}(X;\Vh_*)
\end{equation*}   
where we recall that the notation $\widetilde F$ has been introduced in \eqref{eq_def_f_tilde}.
Let $\gamma\colon E_3\to E_2$ be a smooth splitting. 
This yields a smooth isomorphism of bundles $u=(\alpha, \gamma) \colon E_1\oplus E_3\to E_2$. 
The inverse $u^{-1}$ has the form 
\[
u^{-1} = \begin{pmatrix}
\sigma \\ \beta 
\end{pmatrix}
\]
where $\sigma$ is a left-inverse of $\alpha$. 
We choose holomorphic coordinates for each $E_j$ over an open $U_i\subset X$. We then get the following equations of matrix valued forms:
\[
u_i \cdot u_i^{-1} = \begin{pmatrix}\alpha_i &  \gamma_i \end{pmatrix} \cdot 
\begin{pmatrix} \sigma_i \\ \beta_i \end{pmatrix} = 1,
\]
 and
 \[
 \begin{pmatrix} \sigma_i \\ \beta_i \end{pmatrix} \begin{pmatrix}\alpha_i &  \gamma_i \end{pmatrix} 
 = \begin{pmatrix} 1 & 0 \\ 0 & 1 \end{pmatrix}.  
 \]
Since $D_2$ is a Bott connection, it is represented by a matrix $\Gamma^{2}_i$ 
with coefficients in $F^1$. 
Note that, since $\gamma$ and $\sigma$ may not be holomorphic, $\Delta^{2} := u^*D_2$ may not be a Bott connection.
However, locally on $U_i$, $\Delta^{2}$ takes the form
\[
\Delta^{2}_i = u_i^{-1}du_i + u_i^{-1}\Gamma^{2}_i u_i.
\]
We have $u_i^{-1}\Gamma^{2}_i u_i \in F^1$, since $\Gamma_i^{2}$ is in $F^1$. 
The matrix $u_i^{-1}du_i$ expands as 
\begin{align*}
\begin{pmatrix} \sigma_i \\ \beta_i \end{pmatrix} \begin{pmatrix}d\alpha_i & d\gamma_i \end{pmatrix} = \begin{pmatrix} \sigma_i d\alpha_i & \sigma_i d\gamma_i \\ \beta_id\alpha_i & \beta_i d\gamma_i. \end{pmatrix}
\end{align*}
Since $\beta_i\gamma_i =1$, we have  $\beta_id\gamma_i = -d\beta_i\gamma_i\in F^1$. Since $d\alpha_i\in F^1$, we see that $u_i^{-1}du_i$ is upper triangular modulo $F^1$.

Now let $\pi \colon [0,1] \times X \to X$ be the projection. 
Let 
$$
\nabla = t\cdot \pi^*\Delta^2 + (1-t)\cdot \pi^*(D_1 \oplus D_3),
$$
and let $\theta_i$ be the connection matrix of $\nabla$ with respect to local coordinates on $U_i\subset X$. 
We continue to use the notion of filtration on $\Ah^*(\pi^*(E_1\oplus E_3))$, $\Ah^*(\pi^*(E_2))$, and $\Ah^*([0,1]\times X;\Vh_*)$ as in the proof of Lemma \ref{lemma:Bott_connections_are_homotopic}.  
We know that $(1-t)\cdot\pi^*(D_1 \oplus D_3)$ is in $F^1$, and we have just shown that $t\cdot \pi^*\Delta^2$ is upper triangular modulo $F^1$. 
Thus, $\theta_i$ is upper triangular modulo $F^1$ as well. 
Hence the local curvature form of $\nabla$, i.e., $\Omega_i = d\theta_i + \theta_i\wedge \theta_i$, is upper triangular modulo $F^1$ as well.  
This implies that $c_i(\nabla)\in F^i\Ah^{2i}([0,1]\times X)$ and hence $K(\nabla)\in F^0\Ah^0([0,1]\times X;\Vh_*)$. 
Now we note that we defined the Chern--Simons form $\KCS(D_1,D_2,D_3)$ as the integral of $K(\nabla')$, and not $K(\nabla)$, for the connection $\nabla'$ on $[0,1]\times \pi^* E_2$ 
given by 
\[
\nabla'=(u^{-1})^*\nabla = t \cdot \pi^*D_2 + (1-t) \cdot \pi^*((u^{-1})^*(D_1\oplus D_3)).
\]
Locally we can express the curvature of $\nabla'$ as
\[
\Omega'_i = u_i^{-1} \Omega_i  u_i.
\]
Thus $\nabla$ and $\nabla'$ have identical Chern--Weil forms. 
In particular, this implies that $K(\nabla')\in F^0\Ah^0([0,1]\times X;\Vh_*)$. 
Thus, again since integrating out $dt$ sends $F^0\Ah^0$ to $F^0\Ah^{-1}$, we have shown 
\[
\pi_* K(\nabla') = \KCS(D_1,D_2,D_3) \in \widetilde{F}^0\Ah^{-1}(X;\Vh_*)
\]
which finishes the proof of the lemma and of Theorem \ref{thm:Bott_connection_gives_an_orientation}. 
\end{proof}


A key application of Theorem \ref{thm:Bott_connection_gives_an_orientation} is that it allows us to make a canonical choice of a Hodge filtered $MU$-orientation for each holomorphic map: 

\begin{defn}
\label{def:canonical_MUD_orientation}
Let $g\colon X \to Y$ be a holomorphic map, and let 
\[
\Nh_g:= [g^*TY] - [TX] \in K_{\hol}^0(X)
\]
denote the virtual holomorphic normal bundle of $g$. 
We define the \emph{Bott $\MUD$-orientation of $g$}, or \emph{Bott orientation of $g$} for short, to be $B(\Nh_g) \in \MUDors(X)$, i.e., the image of $\Nh_g$ under $B \colon K_{\hol}^0(X) \to \MUDors(X)$. 
\end{defn}

The next lemma shows that the Bott orientation is functorial, i.e., it is compatible with pullbacks in the following way:

\begin{lemma}\label{lem:pullback_of_canonical_orientation}
Assume we have a pullback diagram in $\Manc$ 
\begin{align} \label{cartesian_square_close_to_main_theorem}
\xymatrix{
X' \ar[d]_-{g'} \ar[r]^-{f'} & X \ar[d]^-g \\
Y' \ar[r]_-f & Y
}
\end{align}
with $f$ transverse to $g$. 
Let $\Nh_g$ and $\Nh_{g'}$ be the virtual holomorphic normal bundles of $g$ and $g'$, respectively. 
Then we have 
\[
{f'}^*B(\Nh_g) = B(\Nh_{g'}) ~ \text{in} ~ \MUDors(X').
\]
\end{lemma}
\begin{proof}
Since $f$ is transverse to $g$, we have ${f'}^*\Nh_g = \Nh_{g'}$ in $K_{\hol}^0(X')$. 
Since the choice of Bott connection does not matter for $B$ by Theorem \ref{thm:Bott_connection_gives_an_orientation},  
this induces the identity ${f'}^*B(\Nh_g) = B(\Nh_{g'})$ in $\MUDors(X')$
\end{proof}

\begin{remark}\label{rem:Bott_of_composition}
Let $X_1 \xto{g_1} X_2 \xto{g_2} X_3$ be proper holomorphic maps. 
Since the map $K_{\hol}^0(X)\to \MUDors(X)$ is a homomorphism of groups, we have  
\[
B(\Nh_{g_2 \circ g_1}) = B(\Nh_{g_1} \oplus g_1^*\Nh_{g_2}) = B(\Nh_{g_1}) + g_1^*B(\Nh_{g_2}).
\]
Hence the Bott orientation of $g_2 \circ g_1$ is the composed $\MUD$-orientation of the Bott orientations of $g_1$ and $g_2$, respectively. 
Together with Lemma \ref{lem:pullback_of_canonical_orientation} this may justify to call the Bott orientation a \emph{canonical} Hodge filtered $MU$-orientation for a holomorphic map. 
\end{remark}

Applying Theorems \ref{generalPushforwardhfc},  
\ref{thm:pushforward_is_functorial},  
\ref{thm:pushpull}, and \ref{thm:projection_formula} with the Bott orientation together with Remark \ref{rem:Bott_of_composition} 
yields the following result: 

\begin{theorem} \label{thm:canonical_pushforward}
Let $X$ and $Y$ be complex manifolds, and let $g \colon X \to Y$ be a proper holomorphic map of codimension $d$. 
We equip $g$ with its Bott orientation $\MUDor:=B(\Nh_g)$.  
Then $g_*:=g_*^{\MUDor}$ defines a functorial pushforward map 
\begin{align*}
g_* \colon MU^n(p)(X)\to MU^{n+2d}(p+d)(Y).  
\end{align*}
This is a homomorphism of $MU^*(*)(Y)$-modules in the sense that, for all integers $n$, $p$, $m$, $q$, and all elements $x\in MU^n(p)(X)$ and $y\in MU^m(q)(Y)$, we have 
\begin{align*}
g_*\left( g^*y \cdot x \right) = y \cdot g_*x ~ \text{in} ~ MU^{n+m+2d}(p+q+d)(Y).
\end{align*}
Furthermore, if $f\colon Y'\to Y$ is holomorphic and transversal to $g$, letting $f'$ and $g'$ denote the induced maps as in  \eqref{cartesian_square_close_to_main_theorem}, the following formula holds 
\[
f^*\circ g_* = g^{\prime}_*\circ f^{\prime *}.
\]
\end{theorem}

In the remainder of this section we further reflect on the Bott orientation class $B(\Nh_g)$. 
We note that $[\Nh_g] = [g^*TY]-[TX]$ merely is a virtual bundle and, in general, there may not be a \emph{holomorphic} bundle $N_g$ over $X$ which represents $[g^*TY]-[TX]$ in $K_{\hol}^0(X)$. 
We can, however, obtain a representative of the orientation class $B(\Nh_g)$ in $\MUDors(X)$ as follows: 
Let $g \colon X \to Y$ be a holomorphic map and $i \colon X \to \C^k$ a smooth 
embedding. 
We then get a short exact sequence of complex vector bundles of the form 
\begin{align}
    \label{eq:ses_normal_bundle_section_5}
\xymatrix{0 \ar[r] & TX\ar[r]^-{D(g,i)} & g^*TY \oplus \trivbc_X^k\ar[r] & N_{(g,i)}\ar[r] & 0.}
\end{align}

\begin{prop}
Let $X$ be a Stein manifold, $Y$ any complex manifold and $g\colon X\to Y$ a holomorphic map. Then we can represent the virtual normal bundle of $g$, $[g^*TY]-[TX]\in K^0(X)$ by a holomorphic vector bundle on $X$.
\end{prop}
\begin{proof}
Since $X$ is Stein, we can assume $i$ in \eqref{eq:ses_normal_bundle_section_5} to be holomorphic. 
Hence $N_{(g,i)}$ admits a holomorphic structure.
\end{proof}

For general $X$, however, we cannot expect $N_{(g,i)}$ to be holomorphic.
In particular, $N_{(g,i)}$ does not, in general, represent the difference $[g^*TY]-[TX]$ in $K_{\hol}^0(X)$. 
Yet we have the following result which follows from the defining relations in $\MUDors(X)$ (see also Remark \ref{rem:orientation_on_two_of_three}):

\begin{prop}
\label{prop:canonical_MUD_orientation_BN_g} 
With the above notation, let $D_X$ be a Bott connection for $TX$, and $D_Y$ a Bott connection for $TY$. 
Let $\nabla_{(g,i)}$ be a connection on $N_{(g,i)}$. 
We set $\MUDor_g := [N_{(g,i)},\nabla_{(g,i)}, - \KCS(D_X,g^*D_Y \oplus d, \nabla_{(g,i)})]\in \MUDors(X)$. 
Then we have 
\[
\pushQED{\qed} 
B(\Nh_g)= \MUDor_g ~ \text{in} ~ \MUDors(X). ~ \qedhere
\pushQED{\qed} 
\] 
\end{prop}

For a \emph{projective} complex manifold we can 
represent the canonical $\MUD$-orientation in the following way: 

\begin{prop}\label{prop:sigma_projective_morphism}
Let $g \colon X \to Y$ be a proper holomorphic map. Assume that $X$ is a projective complex manifold. 
Then there is a holomorphic vector bundle $N$ on $X$ and a Bott connection $D$ on $N$ such that $(N,D,0)$ is an $\MUD$-orientation of $g$ and $B(\Nh_g) = [N,D,0]$ in $\MUDors(X)$.      
\end{prop}

\begin{proof} 
Recall the Euler sequence 
\[
0 \longrightarrow \trivbc \longrightarrow \gamma_1^{\oplus(n+1)} \longrightarrow T\CP^{n} \longrightarrow 0
\]    
where $\gamma_1\to \CP^n $ is the tautological line bundle. 
There is a canonical inclusion $\gamma_1\to \trivbc^{n+1}$, and we denote the quotient by $\gamma_1^\perp$. 
Hence $-[\gamma_1]=[\gamma_1^\perp]-[\trivbc^{n+1}]$.  
Thus we obtain the identity 
\begin{align*}
-[T\mathbb{CP}^n]&=(n+1)\cdot ([\gamma_1^\perp]-[\trivbc^{n+1}])+ [\trivbc] 
= (n+1)\cdot [\gamma_1^\perp] - [\trivbc^{n^2+2n}]
\end{align*} 
in $K^0_{\hol}(\CP^n)$. 
Now let $X$ be a projective manifold and let $\iota \colon X \into \CP^n$ denote a holomorphic embedding. 
We have a short exact sequence of holomorphic vector bundles over $X$ 
\[
0 \longrightarrow TX \longrightarrow \iota^*T\CP^n \longrightarrow NX \longrightarrow 0.
\]     
In $K^0_{\hol}(X)$ this implies the identities
\[
-[TX] = [NX]-\iota^*[T\CP^n] = 
[NX] + (n+1)\iota^*[\gamma_1^\perp]- [\trivbc_X^{n^2+2n}] 
\]
and hence  
\begin{align*}
\Nh_g = [g^*TY]-[TX] 
= [g^*TY] + [NX] + (n+1)\iota^*[\gamma_1^\perp] - [\trivbc_X^{n^2+2n}]. 
\end{align*} 
We define the holomorphic bundle $N:= g^*TY \oplus NX \oplus \iota^*(\gamma_1^\perp)^{\oplus(n+1)}$. 
Since $B(\trivbc^{n^2+2n})=0$, we then get the  identity $B(\Nh_g) = B(N)$ in $\MUDors(X)$. 
Thus we have $B(\Nh_g) = [N,D,0]$ for any Bott connection $D$ on $N$. 
\end{proof}


\section{Fundamental classes and secondary cobordism invariants}
\label{sec:fund_and_secondary_classes}

The existence of pushforwards along proper holomorphic maps allows us to define special types of Hodge filtered cobordism classes. 
In particular, we can define fundamental classes as follows: 

\begin{defn}\label{def:fund_class}
Let $f \colon Y \to X$ be a proper holomorphic map of codimension $d$. 
Let $1_Y \in MU^0(0)(Y)$ be the identity element of the graded commutative ring $MU^*(*)(Y)$. 
We endow $f$ with its Bott orientation. 
We then refer to the element $[f] := f_*(1_Y) \in MU^{2d}(d)(X)$ as the \emph{fundamental class of $f$}. 
If the context of $f$ and $X$ is clear, we may also write $[Y]$ for $[f]$ and call it the \emph{fundamental class of $Y$}. 
\end{defn}

Let $f \colon Y \to X$ be a proper holomorphic map of codimension $d$. 
Let $i \colon Y\to \C^k$ be a smooth 
embedding. 
We then get a short exact sequence of the form 
\begin{align}\label{eq:ses_normal_bundle}
\xymatrix{0 \ar[r] & TY\ar[r] & f^*TX \oplus \trivbc_Y^k\ar[r] & N_{(f,i)}\ar[r] & 0.}
\end{align}
With this notation, we have the following result: 

\begin{prop}
\label{thm:fund_class_explicit}
The fundamental class $[f]$ of $f$ in $MU^{2p}(p)(X)$ is given by 
\begin{align*}
f_*[1_Y] = \left[\geocycle, f_*\sigma_{(f,i)}\right] = \left[f, N_{(f,i)}, \nabla_{(f,i)}, f_*\sigma_{(f,i)}\right] 
\end{align*}
where $\nabla_{(f,i)}$ is any connection on $N_{(f,i)}$ and $\sigma_{(f,i)} = - \KCS\left(D_Y, f^*D_X\oplus d,\nabla_{(f,i)}\right)$ for Bott connections $D_X$ on $TX$ and $D_Y$ on $TY$.  
\end{prop}
\begin{proof}
This follows directly from the description of the Bott orientation in Proposition \ref{prop:canonical_MUD_orientation_BN_g} and the definition of the pushforward map using $1_Y=[\id_Y,d,0]$. 
\end{proof}


Next we show that the fundamental class is compatible with products in the following sense: 

\begin{lemma}\label{lemma:product_of_fund_classes}
Let $f \colon Y \to X$ and $g \colon Z \to X$ be proper holomorphic maps of codimension $d$ and $d'$, respectively. 
Let $\pi$ denote the map induced by the following cartesian diagram in $\Manc$
\begin{align*}
\xymatrix{
Y \times_X Z \ar[d]_-{g'} \ar[r]^-{f'} \ar@{.>}[dr]^-{\pi} & Z \ar[d]^-g \\
Y \ar[r]_-f & X.
}
\end{align*}
Assume that $f$ and $g$ are transverse. 
Then we have 
\begin{align*}
[f] \cdot [g] = [\pi] ~ \text{in} ~ MU^{2d+2d'}(d+d')(X). 
\end{align*}
\end{lemma}
\begin{proof}
Since $f$ and $g$ are transverse, we can apply 
Theorem \ref{thm:pushpull} to get 
$f^*g_* = g'_*{f'}^*$. 
Since $\pi = g' \circ f$ by definition, Theorem \ref{thm:pushforward_is_functorial} implies 
\begin{align*}
f_*f^*g_* = f_*g'_*{f'}^* = \pi_*{f'}^*. 
\end{align*}
We apply this to $1_Z \in MU^0(0)(Z)$ and use that ${f'}^*(1_Z) = 1_{Y\times_X Z}$ to get 
\begin{align*}
[\pi] = \pi_* {f'}^*(1_Z) = f_*f^*g_*(1_Z) = f_*f^*[g]. 
\end{align*} 
Now we apply Theorem \ref{thm:projection_formula} to $y=[g]$ and $x=1_Y$ to conclude 
\begin{align*}
[\pi] = \pi_*(1_{Y \times_X Z}) = f_*f^*[g] = [g] \cdot f_*(1_Y) = [g] \cdot [f]. 
\end{align*}
Finally, we note that the product in the subring of even cohomological degrees $MU^{2*}(*)(X)$ is commutative to conclude the proof. 
\end{proof}

\begin{remark}\label{rem:fund_class_of_submanifold}
If $f \colon Y \into X$ is the embedding of a complex \emph{submanifold} of codimension $d$, then the normal bundle $N_f$ is a holomorphic bundle. 
Hence, in this case, the Bott orientation of $f$ is given by $B(N_f) = (N_f,D_f,0)$ with a Bott connection $D_f$ on $N_f$, and we have $[f] = [f, N_{f}, D_{f}, 0]$ in $MU^{2d}(d)(X)$. 
\end{remark}

\begin{remark}\label{rem:geom_bordism_relation_not_enough}
Let $f_0 \colon Y_0 \to X$ and $f_1 \colon Y_1 \to X$ be two embeddings of complex submanifolds of codimension $d$. 
By Remark \ref{rem:fund_class_of_submanifold} we can write the associated fundamental classes as $[f_0]=[f_0,N_{f_0},D_{f_0},0]$ and $[f_1]=[f_1,N_{f_1},D_{f_1},0]$. 
Now assume that $f_0$ and $f_1$ are cobordant, i.e., they represent the same element in $MU^{2d}(X)$. 
Then we can find a geometric bordism $\geocob$ with $\partial\geocob = \geocycle_1-\geocycle_0$. 
The bordism $\geocob$ is, in general, not sufficient to show $[f_0] = [f_1]$ in $MU^{2d}(d)(X)$, since the associated current $\psi(\geocob)$ defined in \eqref{eq:GeometricBordismDatum} may not vanish.
In fact, $\geocob$ defines a Hodge filtered bordism datum between $f_0$ and $f_1$ if and only if 
\[
\psi(\geocob) \in \widetilde{F}^d \Ds^{2d-1}(X;\Vh_*) = F^d \Ds^{2d-1}(X;\Vh_*) + d\Ds^{2d-2}(X;\Vh_*).
\]
In particular, two homotopic maps $f_0$ and $f_1$ do not define the same class in Hodge filtered cobordism in general (see also Lemma \ref{bordismCycleDependenceOnConnection} and \cite[Lemma 5.9]{ghfc}). 
This shows that the current $\psi(\geocob)$ contains information that is not detected by $MU^{2d}(X)$.  
\end{remark}


Following Remark \ref{rem:geom_bordism_relation_not_enough} we will now study the case of a \emph{topologically} cobordant fundamental class in more detail. 
For the rest of this section we assume that $X$ is a \emph{compact K\"ahler} manifold.  
Then we can split the long exact sequence of 
Proposition \ref{currentiallongexactseq} into a short exact sequence as follows. 
Let $\HdgMU^{2p}(X) = I(MU^{2p}(p)(X))$. 
We write  
\begin{align*}
J_{MU}^{2p-1}(X) = \frac{H^{2p-1}\left(X;\frac{\Ds^*}{F^p}(\Vh_*)\right)}
{\phi(MU^{2p-1}(X))}. 
\end{align*}
Then we get a short exact sequence 
\begin{align} \label{fundamentalShortExactSequenceAbelJacobiSection}
\xymatrix{
0\ar[r] & J^{2p-1}_{MU}(X) \ar[r] & MU^{2p}(p)(X) \ar[r] & \HdgMU^{2p}(X) \ar[r] & 0.
}
\end{align}

\begin{remark}\label{rem:Jacobian_iso}
Note that, since $X$ is compact K\"ahler, we have an isomorphism 
\begin{align*}
H^{2p-1}\left(X;\frac{\Ds^*}{F^p}(\Vh_*)\right) \cong 
\frac{H^{2p-1}(X;\Vh_*)}{F^{p}H^{2p-1}(X;\Vh_*)}. 
\end{align*}
Thus we can rewrite $J_{MU}^{2p-1}(X)$ as 
\begin{align*}
J_{MU}^{2p-1}(X) 
= 
\frac{H^{2p-1}(X;\Vh_*)} {F^pH^{2p-1}(X;\Vh_*)+\phi(MU^{2p-1}(X))}. 
\end{align*}
\end{remark}

\begin{remark}\label{rem:Jacobian_is_flat}
As noted in \cite[Remark 4.12]{hfc}, it follows from the Hodge decomposition that $J_{MU}^{2p-1}(X)$ is isomorphic to the group $MU^{2p-1}(X) \otimes \R/\Z$. 
This implies that, as a \emph{real} Lie group, $J_{MU}^{2p-1}(X)$ is a homotopy invariant of $X$, while as a \emph{complex} Lie group $J_{MU}^{2p-1}(X)$ depends on the complex structure of $X$. 
\end{remark}


\begin{defn}\label{def:secondary_invariants} 
Assume we have an element $[\gamma]$ in $MU^{2p}(p)(X)$ such that $I([\gamma])$ vanishes in $MU^{2p}(X)$. 
Then sequence \eqref{fundamentalShortExactSequenceAbelJacobiSection} shows that we may use $J_{MU}^{2p-1}(X)$ as the target for \emph{secondary cobordism invariants}. 
For example, 
let $f \colon Y \to X$ be a proper holomorphic map of codimension $p$. 
Assume that the fundamental class of $f$ in $MU^{2p}(X)$, given as the pushforward of $1_Y \in MU^0(Y)$ along $f$, vanishes. 
Then the fundamental class of $f$ in $MU^{2p}(p)(X)$ has image in the subgroup $J_{MU}^{2p-1}(X)$. 
Because of the similarity to the Abel--Jacobi map of Deligne--Griffiths (see e.g.\,\cite[\S 12]{voisin1}) 
we will denote the image of $f$ in the subgroup $J_{MU}^{2p-1}(X)$ by $AJ(f)$ and will refer to $AJ(f)$ as the \emph{Abel--Jacobi invariant} of $f$. 
\end{defn}

Let $f \colon Y \to X$ be a proper holomorphic map of codimension $p$. 
We will now describe $AJ(f)$ in more detail. 
Let $[\gamma_f]:=[f,N_{(f,i)},\nabla_{(f,i)},f_*\sigma_{(f,i)}]$ be as in Proposition \ref{thm:fund_class_explicit}. 
We assume that  $f_*(1_Y)=0$ in $MU^{2p}(X)$. 
Then there is a topological bordism datum $b \colon W \to \mathbb R \times X$ such that $\partial b = f$. 
Let $N_b$ be the associated normal bundle. 
We can extend the connection $\nabla_{(f,i)}$ on $N_{(f,i)}$ to get a connection $\nabla_b$ on $N_b$, and obtain a geometric cobordism datum $\geocob$. 
Then we have 
\[ 
\gamma_f - (\partial \widetilde b, \psi(\geocob)) = (0, f_*\sigma_{(f,i)} - \psi(\geocob )) = \left(0, f_*\sigma_{(f,i)} - \left(\pi_X\circ b|_{W_{[0,1]}}\right)_* \left(K^p(\nabla_b)\right)\right) 
\] 
by definition of $\psi(\geocob)$ in \eqref{eq:GeometricBordismDatum}. 
Hence we get 
\[
[\gamma_f]= a\left[f_*\sigma_{(f,i)} - \psi(\geocob)\right]
\]
under the homomorphism  
\[
a \colon H^{2p-1}\left(X;\frac{\Ds^*}{F^p}(\Vh_*)\right) \to MU^{2p-1}(p)(X)
\]
which is induced by the map defined in \eqref{eq:a_forms_def}. 
The class $\left[f_*\sigma_f-\psi(\geocob)\right]$ in $H^{2p-1} \left( X;\frac{\Ds^*}{F^p}(\Vh_*) \right)$ 
may depend on the choice of $\geocob$. 
However, if $\geocob'$ is a different bordism datum, then we have 
\begin{align*}
a\left[\psi(\geocob) - \psi(\geocob')\right] \in \phi(MU^{2p-1}(X)) \subset H^{2p-1}(X;\Vh_*). 
\end{align*}
Thus, after taking the quotient, we get a well-defined class. 
We summarise these observations in the following theorem. 

\begin{theorem}\label{thm:AJ(f)_Kahler}
With the above assumptions on $f$ and $X$, the fundamental class of $f$ in $MU^{2p}(p)(X)$ is the image of 
\begin{align*}
AJ(f) = \left[f_*\sigma_{f} - \psi(\geocob)\right] 
\in \frac{H^{2p-1}(X;\mathcal V_*)} {F^pH^{2p-1}(X;\mathcal V_*)+\phi(MU^{2p-1}(X))} = J^{2p-1}_{MU}(X). \qed
\end{align*}
\end{theorem}


Now we give an alternative description of $AJ(f)$. 
Let $\Vh_*'$ be the $\C$-dual graded algebra with homogeneous components
\[ 
\Vh_j' = (\Vh_{-j})'=\Hom_{\C}(\Vh_{-j},\C).
\] 
Then the canonical pairing given by evaluation 
\[
\ev \colon \Vh_*'\otimes \Vh_* \to \C
\] 
has degree $0$, if $\C$ is interpreted as a graded vector space concentrated in degree $0$. 
Let $n = \dim_{\C} X$. 
Poincar\'e duality and the fact that all vector spaces involved are finite-dimensional imply that the pairing 
\begin{align} \label{perfectPairing}
H^k(X;\mathcal V_*) \times H^{2n-k}(X;\mathcal V_*') & \longrightarrow \mathbb C\\
\nonumber \langle [\eta], [\omega]\rangle & = \ev\left(\int_X \eta\wedge \omega \right)
\end{align}
is perfect. 
Here $\eta\wedge\omega$ is interpreted as a $\mathcal V_*\otimes \mathcal V_*'$-valued form. 
We may thus identify $H^{2p-1}(X;\mathcal V_*)$ with $\left( H^{2n-2p+1}(X;\mathcal V_*')\right)'$. 
Hodge symmetry and Serre duality then imply that, 
under this identification, the subspace 
$F^{p}H^{2p-1}(X;\mathcal V_*)$  
corresponds to 
$\left(F^{n-p+1}H^{2n-2p+1}(X;\mathcal V_*')\right)^\perp$. 
This implies that there is a natural isomorphism 
\[ 
\frac{H^{2p-1}(X;\mathcal V_*)}{F^{p}H^{2p-1}(X;\mathcal V_*)} \cong \left(F^{n-p+1}H^{2n-2p+1}(X;\mathcal V_*')\right)'. 
\]  
Now we let $\phi'$  
denote the composition of $\phi \colon MU^{k}(X)\to H^{k}(X;\Vh_*)$ followed by the identification under pairing \eqref{perfectPairing}, i.e., 
$\phi'$  maps the element  
$[f \colon Z \to X]\in MU^{k}(X)$ 
to $\phi'(f)$ in $\left(H^{2n-k}(X;\mathcal V_*')\right)'$ defined by
\begin{align*}
\phi'(f)([\omega]) := \ev \left(\int_Z  K(\nabla_f) \wedge f^*\omega \right)
\end{align*}
where $\nabla_f$ is a connection on the normal bundle $N_f$. 
We note that, since $Y$ is closed, it follows from Stokes' theorem that this pairing is independent of the choice of representative of $[\omega]$ and of the choice of connection. 
In fact, it is independent of the choice of form which represents the class $K(N_f)$.  
Then we conclude from the above arguments that 
there is a natural isomorphism 
\begin{align}\label{eq:iso_JMU}
J^{2p-1}_{MU}(X) \cong  \frac{\left(F^{n-p+1}H^{2n-2p+1}(X;\mathcal V_*')\right)'}{\phi'(MU^{2p-1}(X))}. 
\end{align}

Now let $f \colon Y \to X$ be a proper holomorphic map of codimension $p$ 
such that $[f]=0$ in $MU^{2p}(X)$. 
Let $\geocob=(b,N_b,\nabla_b)$ be a geometric bordism datum over $b=(c_b,f_b)\colon W\to \R\times X$. 
We set $W_{[0,1]}:=c_b^{-1}([0,1])$, and $w:=f_b|_{W_{[0,1]}}$. 

\begin{theorem}\label{thm:AJ_alternative}
With the above notation, the image of $AJ(f)$ under isomorphism \eqref{eq:iso_JMU} is represented by the functional in $\left(F^{n-p+1}H^{2n-2p+1}(X;\Vh_*')\right)'$ defined by 
\[
[\omega]\mapsto \ev \left(  \int_{Y}\sigma_f \wedge f^*\omega + \int_{W_{[0,1]}} K(\nabla_b) \wedge w^*\omega \right). 
\]
\end{theorem}
\begin{proof}
We recall from Theorem \ref{thm:AJ(f)_Kahler} that $AJ(f) = \left[f_*\sigma_f-\psi(\geocob)\right] \in J_{MU}^{2p-1}(X)$. 
Let $\omega$ be a closed form in $F^{n-p+1}\Ah^{2n-2p+1}(X;\Vh'_*)$. 
Since the codimension of $f_b$ is odd, we have  $\psi(\geocob)=-w_*K(\nabla_b)$.  
Then the interaction of pushforwards and pullbacks with integrals and Stokes' theorem yield:
\begin{align*}
     \int_X \left(f_*\sigma_f-\psi(\geocob)\right) \wedge \omega 
    & =  \int_X f_*\sigma_f\wedge \omega- \int_X\psi(\geocob)\wedge \omega \\
    & =  \int_{Y}\sigma_f \wedge f^*\omega + \int_{ W_{[0,1]}}K(\nabla_b)\wedge w^*\omega. 
\end{align*}
By construction of isomorphism \eqref{eq:iso_JMU}, the image of $AJ(f)$ is the homomorphism that sends $[\omega]$ to the class given by evaluating the above sum of integrals. 

It remains to show that this evaluation yields a well-defined element in the group $\left(F^{n-p+1}H^{2n-2p+1}(X;\Vh_*')\right)'$. 
Assume $\omega = d\psi$. 
Then 
\begin{align}\label{eq:1st_int_over_Y_is_well-defined}
\int_{ W_{[0,1]}}K(\nabla_b)\wedge w^*(d\psi)
= \int_{Y} K(\nabla_b)_{|Y} \wedge f^*\psi
\end{align}
by Stokes' theorem. 
Since $f$ is holomorphic, we have $K(N_b)|_Y=K(N_f)\in H^{0,0}(Y;\Vh_*)$ and thus $K(\nabla_b)_{|Y} \in F^0\Ah^0(Y;\Vh_*)$.  
Since Hodge theory implies the vanishing $F^{n-p+1}H^{2n-2p}(Y;\Vh_*\otimes \Vh'_*)=0$, integral \eqref{eq:1st_int_over_Y_is_well-defined} vanishes. 

For the other integral we note that by Stokes' theorem we have 
\begin{align}\label{eq:2nd_int_over_Y_is_well-defined}
 \int_{Y}\sigma_f \wedge f^*(d\psi) = 
  \int_{Y} d\sigma_f \wedge f^*\psi.
\end{align}
We recall from Proposition \ref{thm:fund_class_explicit} that $\sigma_f = - \KCS\left(D_Y, f^*D_X\oplus d,\nabla_f\right)$ for Bott connections $D_X$ on $TX$ and $D_Y$ on $TY$, and an arbitrary connection $\nabla_f$ on the normal bundle. 
The derivative of $\sigma_f$ satisfies 
\[
d\sigma_f = K(f^*D_X\oplus d) - K(D_Y) - K(\nabla_f).
\]
Since $K$ is multiplicative and $K(d)=1$, we have $K(f^*D_X\oplus d) = K(f^*D_X)$. 
Since $D_X$ and $D_Y$ are Bott connections, we know 
that $K(f^*D_X)$ and $K(D_Y)$ are in $F^0\Ah^0(Y;\Vh_*)$. 
This implies again for reasons of type that the integrals 
\[
 \int_{Y}K(f^*D_X) \wedge f^*\psi \ \  \text{and} ~  \int_{Y}K(D_Y) \wedge f^*\psi
\]
both vanish. 
The remaining term to analyse is the integral 
$\int_{Y} K(\nabla_f) \wedge f^*\psi$ 
which we already have shown to vanish. 
Thus integral \eqref{eq:2nd_int_over_Y_is_well-defined} vanishes and the functional is well-defined. 
Finally, we note that integral \eqref{eq:2nd_int_over_Y_is_well-defined} is independent of the chosen bordism datum, while  
the difference between the integrals \eqref{eq:1st_int_over_Y_is_well-defined} corresponding to two different bordism data is an element in $\phi'(MU^{2p-1}(X))$. 
\end{proof}

\begin{remark}\label{rem:sigma_term_vanishes_in_AJ_if}
The formula in Theorem \ref{thm:AJ_alternative} simplifies if the orientation $\MUDor_f$ admits a representative of the form $(N,\nabla,0)$. If $f$ is projective, we obtain such a representative from Proposition \ref{prop:sigma_projective_morphism}, and if $f$ is a holomorphic embedding, $B(f^*TX/TY)$ will do. 
We do not know if such representatives exist for the canonical orientations of general holomorphic maps. 
\end{remark}


\section{Hodge filtered Thom morphism} \label{section:hfThomMorphism}

We will now define a Thom morphism from Hodge filtered cobordism to Deligne cohomology.  
In order to define a map on the level of cycles we will first construct a new cycle model for Deligne cohomology. 
Our construction is similar to that of Gillet--Soul\'e in \cite{GilletSoule} (see also \cite{HarveyLawson}). 
However, our construction is more elementary than the one in \cite{GilletSoule} in the sense that it avoids the use of geometric measure theory.

Let $X$ be a complex manifold and $U \subseteq X$ an open subset. 
For an integer $p\ge 0$, let $\Z(p)$ denote  $(2\pi i)^p\cdot\Z$ and let $\Z_{\Dh}(p)$ be the complex of sheaves 
\begin{align*}
0 \to \Z(p) \to \Oh_X \to \Omega_X^1 \to \cdots \to \Omega_X^{p-1} \to 0
\end{align*} 
where $\Z(p)$ is placed in degree $0$. 
Then the Deligne cohomology group $\HD^q(X;\Z(p))$ may be defined as the $q$-th hypercohomology of the complex $\Z_{\Dh}(p)$. 
We recall the group of smooth relative chains defined as the quotient
\[
C_{\dim X-k}^{\diff}(X, X\backslash \overline U;\mathbb Z(p))=\frac{C_{\dim X-k}^{\diff}(X;\Z(p))}{C_{\dim X-k}^{\diff}(X\backslash \oU;\Z(p))}. 
\]
Let $\overline C^k$ denote the presheaf 
\[
U\mapsto \oC^{ k}(U):=C_{\dim X-k}^{\diff}(X, X\backslash \oU;\mathbb Z(p)).
\] 
The restriction maps of $\overline C^k$ are induced by quotienting out the appropriate additional chains. 
The presheaf $\oC^k$ is very close to being a sheaf since it satisfies the sheaf condition for coverings of $X$. 
However, it does not satisfy the sheaf condition for general collections of open subsets of $X$. 
Hence let $C^k$ be the sheafification of $\overline C^k$. 
The sheaf $C^k$ is not fine, but it is homotopically fine, meaning that its endomorphism sheaf admits a homotopy partition of unity. 
We refer to \cite[page 172]{Bredon}, from which we also recall the implication that $H^*(H^j(\oC^*(U)))=0$ for $j>0$.  
Hence the hypercohomology spectral sequence degenerates on the $E_2$-page, past which only the row $H^0(C^*(U))$ survives. 
On stalks the sheaf $C^k$ coincides with the presheaf $\overline C^k$. 
Let $U$ be a small contractible open subset of $X$. 
By excision we have 
\[
H^k(\oC^*(U))=H_{\dim X-k}(\R^{\dim X}, \R^{\dim X}\setminus \mathbb{D};\Z(p)),
\]
for $\mathbb{D}$ the closed unit disc. 
Hence we get
\[
H^j(\overline C^*(U)) = 
\left\{\begin{array}{cc}
     0 & j>0  \\
     \Z(p)& j=0. 
\end{array}\right.
\]
This proves the following result: 
\begin{lemma}\label{acyclicResolutionOfZp}
The complex $C^*$ is an acyclic resolution of the constant sheaf $\Z(p)$ as sheaves on $X$. \qed 
\end{lemma}

By \cite[Appendix B, I.12]{Bredon} we also have the following fact: 

\begin{lemma}\label{SingularCohomologyFromChainComplex}
The canonical map $\overline C^*\to C^*$ induces an isomorphism of cohomology groups on global sections
 $H^k(\overline C^{*}(X))=H^k(C^*(X))$.\qed
\end{lemma}  
In other words, the sheaf cohomology $H^k(X;\Z(p))$ can be computed as the cohomology of the complex $\oC^{*}(X)$. 
Now we consider the map of complexes 
\[
T\colon \overline C^*(X)\to \Ds^*(X)
\]
induced by integration. 
Let $\Ds^*_\Z(X)$ be the image of $T$ in $\Ds^*(X)$. 
Since $T$ is a map of chain complexes, it follows that $\Ds^*_\Z(X)$ is a complex as well.

\begin{prop}
The map $T\colon \overline C^*(X)\to \Ds^*_\Z(X)$ induces an isomorphism on cohomology.
\end{prop} 
\begin{proof}
By Whitehead's triangulation theorem, we may pick a smooth triangulation of $X$, i.e., a set $S=\{f_i\colon \Delta^{k_i}\to X\}$ such that each $f_i$ is a continuous embedding which extends to a smooth mapping of a neighborhood of $\Delta^k\subset \R^k$, and each $x\in X$ is in the interior of a unique cell $S_i=\Imm(f_i)$. 
It is well-known that the inclusion of cellular chains $C_*(S;\Z(p))\to C_*(X;\Z(p))$ is a quasi-isomorphism. 
Hence it suffices to show that $T$ restricts to a quasi-isomorphism on the cellular chains of $S$. 
Since each point $x\in X$ is contained in the interior of a unique cell of $S$, we can show that $T$ is injective on cellular chains as follows. 
We can construct for each $i$ a form $\omega_i\in \Ah^{k_i}(X)$ such that $\int_{\Delta^{k_i}}f_i^*\omega_i\neq 0$, and such that the only $k_i$-cell intersecting the support of $\omega_i$ is $S_i$. Suppose $T(c)=0$ for $c=\sum a_if_i$. 
Then $T(c)(\omega_i) = a_iT(f_i)(\omega_i)$ is a nonzero multiple of $a_i$, and we get $a_i=0$ for all $i$.
To see that the map induced by $T$ from cellular homology is injective, we first note that the inclusion of cellular chains into singular chains is a 
deformation retract since it is a quasi-isomorphism between complexes of projective modules. 
Let $r$ be a retraction onto the cellular chains. 
Now let $c$ be a cellular cycle with $T(c)=dT(\alpha)$ for $\alpha$ an arbitrary integral chain $\alpha \in C^*(X)$. 
Then we have $T(c) = dT(\alpha) = T(\partial \alpha)$ and thus $T(c)=T(r(c)) = T(r(\partial\alpha)) = T(\partial r(\alpha))$. 
Since $T$ is injective on cellular chains, we get $c=\partial r(\alpha)$. 
Hence $c$ represents $0$ in cellular homology, and the map induced by $T$ on cellular homology is injective.
It remains to see that $T$ restricted to cellular chains is surjective on homology. 

By definition of $\Ds^*_\Z(X)$ as the image of $T$, every element of $\Ds^*_\Z(X)$ is of the form 
$\sum_i T(a_i\cdot g_i)$ where $g_i$ are smooth maps $\Delta^k \to X$. 
Assume that $\sum_i T(a_i\cdot g_i)$ is a cycle and hence represents a class in $H^k(X;\Ds_\Z)$. 
To simplify the notation, we write $g:= \sum_ia_i\cdot g_i$. By assumption, we have $dT(g)=0$. 
Since $r$ is a deformation retraction, there is a homotopy $h$ of the cellular chains such that 
\[
\partial h + h \partial = 1-r.
\]
By applying $r$, we define a cellular chain $f:=r(g)$. 
Omitting the inclusion from cellular chains into chains from the notation we then have the identity of chains 
\[
g':= f - h\partial(g) = \partial (h(g))+ g. 
\]
Applying $r$ again defines a cellular chain $r(g')$ such that 
\[
dT(r(g')) = dT(r(\partial(h(g)) + g)) = T(\partial\partial (h(g))) + dT(g) = 0 
\]
where we use the assumption $dT(g) = 0$. 
Hence we get $T(\partial(r(g'))) = dT(r(g')) = 0$. 
Since $T$ is injective on cellular chains, this implies $\partial r(g')=0$, i.e., that $f':= r(g')$ is a cellular cycle. 
Since $T(f') - T(g) = T(\partial(h(g))) = dT(h(g))$ is an exact current, we have found a cellular cycle $f'$ whose homology class is mapped to the homology class of $g$ under $T$. 
This completes the proof. 
\end{proof}


We are now ready to give our presentation of Deligne cohomology. Let 
$$
i_F\colon F^p\Ah^*\to \Ds^* 
$$ 
be the map of sheaves induced by $T$, and let $i_c\colon \Ds_\Z^*(X)\to \Ds^*(X)$ be the inclusion.
We will show that the following cochain complex 
\[
C_\Dh^*(p)(X) = \cone\left(\xymatrix{\Ds_{\Z}^*(X)\oplus F^p\Ah^*(X) \ar[r]^-{i_c-i_F} & \Ds^*(X)}\right) 
\]
computes the Deligne cohomology of $X$. 
In degree $k$ we have the group
\[
C^k_\Dh(p)(X) =  \Ds_\Z^k(X)\oplus F^p\Ah^k(X)\oplus \Ds^{k-1}(X).
\]
The differential is defined by
\[
d( T ,\omega, h) = ( dT,\ d\omega,\  i_c(T)-dh+i_F(\omega)).
\]

\begin{theorem}\label{thm:computing_Deligne}
The cohomology of the cochain complex $C^*_\Dh(p)(X)$ is naturally isomorphic to Deligne cohomology. 
\end{theorem}

To prove the theorem we will use multicomplexes, which are more flexible than bicomplexes. 
We recall from \cite{Boardman} that a multicomplex of abelian groups consists of the data of a bigraded abelian group, $E^{s,t}$, and differentials $d_r^{s,t}\colon E^{s,t}\to E^{s+r,t-r+1}$ such that 
\[
\sum_{i+j=k} d_j^{s+i,t-i+1}\circ d_{i}^{s,t} = 0 \colon E^{s,t}\to E^{s+k,t-k+2}.
\]
One can consider multicomplexes of objects in any abelian category. 
We are considering here multicomplexes of abelian sheaves.

\begin{proof}[Proof of Theorem \ref{thm:computing_Deligne}] 
We will construct a series of quasi-isomorphisms of complexes of sheaves
\[
\Z_\Dh(p)\simeq C^{\prime*}_\Dh(p)\simeq \mathrm{Tot}(M)
\]
and a quasi-isomorphism of complexes of abelian groups $\mathrm{Tot}(M)(X)\to C_\Dh^*(X)$,  
where $M$ is the following multicomplex of sheaves on $X$:
\[
M^{s,t}= \begin{cases} 
     C^t & s=0\\
     \Ds^{s-1,t}&  0<s<p\\
     F^p\Ah^{s,t}\oplus \Ds^{s-1,t}& p \leq i. 
\end{cases} 
\]
To define the differentials let $\Pi^{s,k-s} \colon \Ds^k\to \Ds^{s,k-s}$ be the projection. 
For $s>0$, there is only $d_0$ and $d_1$. The differentials of $M$ are 
\begin{align*}
d_0^{s,t} & = \begin{cases} 
    d\colon C^{t}\to C^{t+1} & s=0 \\
    - \overline \partial\colon \Ds^{s-1,t}\to \Ds^{s-1,t+1}  &0<s<p  \\
    (\overline \partial, i_F-\overline\partial) \colon F^p\Ah^{s,t}\oplus\Ds^{s-1,t}\to F^p\Ah^{s,t+1}\oplus \Ds^{s-1,t+1} & 
    s \geq p
\end{cases}
\\
d_1^{s,t} & = \begin{cases}
    \Pi^{0,t} \circ i_c \colon C^{t} \to \Ds^{0,t} & s=0 \\
     -\partial\colon \Ds^{s-1,t}\to \Ds^{s,t} & 0 < s < p  \\
     (\partial,i_F-\partial) \colon F^p\Ah^{s,t}\oplus \Ds^{s-1,t}\to F^p\Ah^{s+1,t}\oplus \Ds^{s,t} & 
     s \geq p 
\end{cases} \\
d_r^{0,t} & = \Pi^{r, t-r}\circ i_c\colon C^t\to \Ds^{r,t-r}. 
\end{align*}
The total complex of $M$ 
is given by 
\[
\mathrm{Tot}^*(M) = \cone\left(\xymatrix{C^*\oplus F^p\Ah^* \ar[r]^-{i_F-i_c} & \Ds^*}\right).
\]
There is therefore a natural map $\mathrm{Tot}^*(M(X))\to C^*_\Dh(p)(X)$ defined by 
\[
\mathrm{Tot}^*(M(X))\ni ( c,\ \omega,\ h)\mapsto (aT(c),\ \omega,\ h)\in C^*_\Dh(p)(X)
\]
where we write $aT$ for the sheafified map induced by $T$.
This map of complexes induces an isomorphism on cohomology  since each of the maps 
\[
\id \colon F^p\Ah^*(X)\to F^p\Ah^*(X),\quad  T\colon \overline C^*(X) \to \Ds_\Z^*(X) \quad \text{and}\quad \id \colon \Ds^*(X) \to \Ds^*(X)
\] 
is a quasi-isomorphism. 
We define yet another complex of sheaves 
\[
C^{\prime*}_\Dh(p) = \left(\Z(p)\xto{} \Omega^0 \xto{d}  \cdots \xto{d} \Omega^{p-2} \xto{(0,d)}  \Omega^p\oplus\Omega^{p-1} \xto{\delta_p} \Omega^{p+1}\oplus\Omega^{p} \xto{\delta_{p+1}} \cdots\right)
\]
with $\delta_i (\omega, \tau)=(d\omega, \omega-d\tau)$ for $i\geqslant p$. 
There is a map $f\colon \Z_\Dh(p)(X) \to C^{\prime*}_\Dh(p)(X)$ given by 
\begin{align*}
\xymatrix{
\Z(p)\ar[d]_{\id}\ar[r] &  \Omega^0\ar[r]^-d \ar[d]_{\id} & \Omega^1\ar[r]^-d \ar[d]_{\id} & \cdots\ar[r]^-d & \Omega^{p-2}\ar[d]^{\id}\ar[r]^-d & \Omega^{p-1}\ar[r]\ar[d]^{\alpha}&0 \\
\Z(p) \ar[r] &\Omega^0\ar[r]^d &\Omega^1\ar[r]& \cdots \ar[r]^d & \Omega^{p-2}\ar[r]^-{(0,d)} &\Omega^{p}\oplus \Omega^{p-1}\ar[r]^-{\delta_p}&\cdots
}
\end{align*}
with $\alpha(\omega) = (d\omega, \omega)$. 
We claim that this is a quasi-isomorphism of complexes of sheaves. 
This is clear in degrees $< p$, and in degrees $>p$ it follows from the fact that $C^{\prime*}_\Dh(p)$ is exact in that range. 
In degree $p$ we need to show that $f$ induces an isomorphism on cohomology of stalks. 
Let $U$ be a polydisc. 
Then 
\[
\HD^p(U;\Z(p)) = \frac{\Omega^{p-1}(U)}{\Imm d},
\]
and 
\[
H^p( U;C^{\prime*}_\Dh(p)) = \frac{\{(\omega,\tau)\in \Omega^p(U)\oplus\Omega^{p-1}(U) \ :\ d\tau = \omega \}}{\Imm (0,d)}.
\]
It is clear that the map induced by $f$, which can be described as $[\tau]\mapsto [d\tau, \tau]$, is an isomorphism. 
Hence $f$ is a quasi-isomorphism as claimed. 
Next there is a natural map $C^{\prime*}_\Dh (p)\to  M$ given by 
\[
\xymatrix{
\mathbb Z(p)\ar[r]\ar[d]^\epsilon & \Omega^0\ar[r]\ar[d]&\cdots\ar[r] & \Omega^{p-2}\ar[r]\ar[d]&\Omega^p\oplus\Omega^{p-1}\ar[r]\ar[d]&\cdots\\
C^0\ar[r]& \Ds^{0,0}\ar[r] & \cdots \ar[r] & \Ds^{p-2,0} \ar[r] & \Ah^{p} \oplus \Ds^{p-1}\ar[r] & \cdots
}
\]
where $\epsilon$ is the quasi-isomorphism $\Z(p)\to C^*$. 
The column $M^{i,*}$ is a resolution of the sheaf $C^{\prime i}_\Dh(p)$ by Lemma \ref{acyclicResolutionOfZp} and the arguments in \cite[pages 382--385]{GriffithsHarris}.  
Hence the natural map $C^{\prime*}_\Dh (p)\to M$ is a quasi-isomorphism. This concludes the proof.
\end{proof}

\begin{remark}
If we choose a smooth triangulation of $Z$, then by summing up the top cells we get a smooth singular cycle $c_Z$ representing the fundamental class $[Z]\in H_{\dim Z}(Z;\Z)$.  
We have $T(c_Z)=1\in \Ds^0(Z)$, and so 
\[
f_*1=f_*T(c_Z)=T(f_*c_Z) \in \Ds_\Z^*(X).
\]
The advantage of using $\Ds_\Z$ is that no choice of triangulation is needed in order to get the current $f_*1$.
\end{remark}

Let $\tau_0$ be the map 
\[
\Ds^*(X;\mathcal V_*)\to \Ds^*(X;\mathbb C)
\]
induced by the map on coefficients $\Vh_*=MU_*\otimes\mathbb C\to \mathbb C$ determined by the additive formal group law over $\C$. 
Then $\tau_0$ is a chain map and it preserves the Hodge filtration. 
Now we are ready to define our Hodge filtered Thom morphism on the level of cycles:
\begin{align*}
\tau_\mathbb Z \colon ZMU^n(p)(X) & \to C^n_\Dh(p)(X), \\
\hfcycle = (\widetilde f, h) & \mapsto (f_*1, \tau_0(R(\hfcycle)), \tau_0(h)).
\end{align*}

\begin{lemma}\label{lemmaP0}
We have $\tau_0(f_*K(\nabla_f)) = f_*1$. 
\end{lemma}
\begin{proof}
This follows from the definition of $\tau_0$ and the fact $K_0=1$ since $K$ is a multiplicative sequence. 
\end{proof}

\begin{theorem}\label{thm:HF_Thom_morphism}
For every $X\in \Manc$, 
the map $\tau_\mathbb Z$ induces a natural homomorphism  
\[
\widehat\tau_\Z \colon MU^n(p)(X)\to H^n_{\Dh}(X;\Z(p))
\]
which fits into a morphism of long exact sequences
\begin{align}\label{ThomMorphismMapOfExactSequences}
\xymatrix{
\cdots \ar[r]& H^{n-1}\left(X;\frac{\Ah^*}{F^p}(\Vh_*)\right) \ar[r] \ar[d]_{\tau_0} & MU^n(p)(X)\ar[r] \ar[d]_{\hat\tau_\Z} & MU^n(X)\ar[r]\ar[d]_{\tau} &\cdots \\
\cdots\ar[r]& H^{n-1}\left(X;\frac{\Ah^*}{F^p}(\C)\right)\ar[r] & H^n_\Dh(X;\Z(p))\ar[r] & H^n(X;\Z)\ar[r] & \cdots
}
\end{align}
\end{theorem}
\begin{proof}
It is clear that $\tau_\mathbb Z$ is a group homomorphism. 
We need to prove that, for a cycle $\hfcycle=(\geocycle,h)\in ZMU^n(p)(X)$, 
we have 
\[
d\tau_\Z(\hfcycle)=0 ~ \text{and} ~ 
\tau_\Z\left(BMU^n(p)(X)\right)\subset dC_\Dh^{n-1}(p)(X).
\]
We begin with the former. 
We have 
\begin{align*}
    d\tau_\mathbb Z(\geocycle,h) 
    & = d(f_*1, \tau_0(R(\hfcycle)), \tau_0(h)) \\ 
    & = (df_*1, \tau_0(dR(\hfcycle)), \tau_0(dh) + f_*1 - \tau_0(R(\hfcycle))). 
\end{align*}
Since $f_*1$ is a closed current, and $R(\hfcycle)$ is a closed form, 
we deduce $d\tau_\mathbb Z(\hfcycle)=0$ from Lemma \ref{lemmaP0}.   
Now let $\widetilde b$ be a geometric bordism datum. 
Then 
\begin{align*}
\widehat\tau_\mathbb Z(\partial \widetilde b, \psi(\geocob))  = (\tau_0\phi(\partial \widetilde b), 0, \tau_0\psi(\widetilde b)) = (\tau _0 d\psi(\widetilde b), 0, \tau_0\psi(\widetilde b)) = d( \tau_0\psi(\widetilde b), 0, 0).
\end{align*}
Next let $h\in \widetilde F^p\Ah^{n-1}(X;\Vh_*)$. 
Then $\tau_0(h)\in \widetilde F^p\Ah^{n-1}(X)$, so that 
\[
(0,\tau_0(h),0)\in C_{\Dh}^{n-1}(p)(X).
\]
We have
\begin{align*}
    \tau_\mathbb Z(a(h)) = \tau_\mathbb Z(0, h)
 = (0, \tau_0(dh), \tau_0(h))  = d(0, \tau_0(h), 0)
\end{align*}        
which finishes the proof that $\tau_\Z$ induces a homomorphism. 
The second assertion follows directly from the construction of $\widehat\tau_\Z$. 
\end{proof}

Let $X$ be a compact K\"ahler manifold. 
Let $f \colon Y \to X$ be the inclusion of a complex submanifold of codimension $p$ such that its fundamental class in $MU^{2p}(X)$ vanishes. 
The latter condition implies that the fundamental class of $f$ in $H^{2p}(X;\Z)$ vanishes as well. 
Hence both the classical Abel--Jacobi invariant $AJ_H(f)$ of Deligne--Griffiths (see e.g.\,\cite[\S 12]{voisin1}) 
and the invariant  $AJ(f)$ of Theorems \ref{thm:AJ(f)_Kahler} and \ref{thm:AJ_alternative} are defined. 

\begin{theorem}\label{thm:AJ_and_Thom_equals_classical}
With the above notation and assumptions, we have  
\[
\tau_0(AJ(f)) = AJ_H(f).
\]
\end{theorem}
\begin{proof}
By Theorem \ref{thm:AJ_alternative} the invariant $AJ(f)$ may be represented by the functional 
\[
[\omega]\mapsto \ev \left(  \int_{Y}\sigma_f \wedge f^*\omega + \int_{W_{[0,1]}} \left( K(\nabla_b) \right) \wedge w^*\omega \right). 
\]
The image of the Chern--Simons form $\sigma_f$ under $\hat\tau_{\Z}$ and $\tau_0$ is zero since $\sigma_f$ is a form in degree $-1$. 
By Lemma \ref{lemmaP0}, $K(\nabla_b)$ is mapped to $1$. 
Thus, $\tau_0$ maps $AJ(f)$ to the class of the functional in $F^{n-p+1}H^{2n-2p+1}(X;\C')'$ defined by 
\[
[\omega]\mapsto \ev  \int_{W_{[0,1]}} w^*\omega. 
\]
This corresponds to the characterization of $AJ_H(f)$ in \cite[\S 12.1.2 on page 294]{voisin1}.
\end{proof}


\section{Image and kernel for compact K\"ahler manifolds}
\label{section:hfThomMorphism_Kahler}

We assume again that $X$ is a compact K\"ahler manifold. 
Then the morphism of long exact sequences \eqref{ThomMorphismMapOfExactSequences} induces a map of short exact sequences 
\begin{align}\label{eq:Thom_morphism_diagram}
\xymatrix{
0 \ar[r] & J_{MU}^{2p-1}(X)\ar[r]\ar[d]^{\tau_J} & MU^{2p}(p)(X)\ar[r]\ar[d]^{\widehat\tau_\mathbb Z} & \mathrm{Hdg}_{MU}^{2p}(X)\ar[d]^{\tau} \ar[r] & 0 \\
0 \ar[r] & J^{2p-1}(X)\ar[r] & H_{\Dh}^{2p}(X;\mathbb Z(p))\ar[r]& \mathrm{Hdg}^{2p}(X) \ar[r] & 0. 
}
\end{align}
Let $\Mh^p(X)$ be the free abelian group generated by isomorphism classes $[f]$ of proper holomorphic maps $f \colon Y \to X$ of codimension $p$. 
For a proper holomorphic map $f \colon Y \to X$ of codimension $p$ we denote its fundamental class in $MU^{2p}(p)(X)$ by $\widehat{\varphi}(f)$ and its fundamental class in $MU^{2p}(X)$ by $\varphi(f)$. 
This defines homomorphisms of abelian groups 
\[
\widehat{\varphi} \colon \Mh^p(X) \to MU^{2p}(p)(X) ~ \text{and} ~ \varphi \colon \Mh^p(X) \to MU^{2p}(X).
\] 
We denote the kernel of $\varphi$ by $\Mh^p(X)_{\topp}$. 
Then the Abel--Jacobi invariant of Definition \ref{def:secondary_invariants} defines a homomorphism 

\[
AJ \colon \Mh^p(X)_{\topp} \to J_{MU}^{2p-1}(X).
\]
Note that every element in $\Mh^p(X)_{\topp}$ is homologically equivalent to zero and therefore has a well-defined image in $J^{2p-1}(X)$. 
By Theorems \ref{thm:HF_Thom_morphism} and \ref{thm:AJ_and_Thom_equals_classical} composition with the respective maps of diagram \eqref{eq:Thom_morphism_diagram} produces the classical invariants. 
Diagram \eqref{eq:Thom_morphism_diagram} shows  that studying the kernel and image of $\widehat{\tau}_{\Z}$ is equivalent to analysing the kernel and image of $\tau_J$ and $\tau$, respectively. 
%
We expect the maps $\widehat{\varphi}$ and $AJ$ to be useful to discover new phenomena and examples that the classical invariants with values in Deligne cohomology are not able to detect.  
We will now briefly report on some results in this direction. \\
%

First we look at the image of $\widehat{\tau}_{\Z}$. 
Let $X$ be a smooth projective complex algebraic variety. 
In \cite{totaro}, Totaro showed that an element in $H^{2*}(X(\C);\Z)$ which is not in the image of $\tau \colon MU^{2*}(X(\C)) \to H^{2*}(X(\C);\Z)$ cannot be algebraic. 
This is a refinement of the obstruction induced by the Atiyah--Hirzebruch spectral sequence (see also \cite{BenoistOttem}). 
It follows from \cite[Corollary 7.12]{hfc} that an algebraic class in $H^{2*}(X(\C);\Z)$ has to be in the subgroup $\tau\left(\HdgMU^{2*}(X(\C)\right)$. 
In \cite[\S 3.4]{Benoist}, Benoist shows that this obstruction to algebraicity of cohomology classes is in fact finer than the one of \cite{totaro}. \\

Now we consider the kernel of $\tau_J$. 
Since $\tau_0$ is an epimorphism of vector spaces, 
the map $\tau_J$ is surjective, and the snake lemma  implies that there is a short exact sequence 
\[
0 \to \ker \tau_J \to \ker \widehat\tau_\mathbb Z \to \ker \tau \to 0.
\] 
Hence $\ker\widehat\tau_\Z$ contains information on the failure of the Thom morphism $\tau$ to be injective on Hodge classes, and on the failure of $\tau_J$ to be injective. 
We have a further short exact sequence 
\begin{align*}
\xymatrix{
0\ar[r] & MU^{2p-1}(X)_{mt} \ar[r]^{\phi_{mt}} \ar[d]^{\tau_{mt}} & \frac{H^{2p-1}(X;\Vh_*)}{F^pH^{2p-1}(X;\Vh_*)} \ar[d]^{\tau_{\overline J}} \ar[r] & J_{MU}^{2p-1}(X)\ar[d]^{\tau_J} \ar[r] & 0 \\ 
0 \ar[r] & H^{2p-1}(X;\mathbb Z)_{mt}\ar[r]^-i & \frac{H^{2p-1}(X;\C)}{F^pH^{2p-1}(X;\C)} \ar[r] & J^{2p-1}(X)\ar[r] & 0
}
\end{align*}
where the subscript $mt$ means modulo torsion.
Again, since $\tau_0$ is onto, it follows that $\tau_{\overline J}$ is onto. 
Then the snake lemma places $\ker \tau_J$ in the exact sequence
\begin{align*}
0 \to \ker \tau \to \ker \tau_{\overline J} \to \ker \tau_J \to \coker \tau_{mt} \to 0.
\end{align*}
This indicates two methods to construct elements in $\ker \tau_J$: 
as elements coming from $\ker \tau_{\overline J}$ or as elements coming from $\coker \tau_{mt}$. 
We will now briefly describe both these methods. \\

The arguments in \cite[\S 7.3]{hfc} show how to construct elements in $\ker \tau_{\overline J}$. 
We note that even though we have not shown that $MU^{2*}(*)(-)$ receives a map from algebraic cobordism for algebraic varieties, we can adjust the arguments as follows. 
Let $\Pro^1$ be the complex projective line, 
and let $[\Pro^1]$ denote corresponding element in $MU^{-2}$. 
Let $f \colon Y \to X$ be a proper holomorphic map of codimension $p$. 
Let $\Pro^1_X \to X$ denote the pullback of $\Pro^1$ to $X$. 
By Lemma \ref{lemma:product_of_fund_classes} we get a well-defined homomorphism 
\begin{align*}
\Mh^p(X) \to MU^{2p-2}(p-1)(X)
\end{align*}
induced by sending $[Y]$ to $[Y] \cdot [\Pro_X^1]$. 
Since $X$ is compact, there is an isomorphism $MU^*(X)\otimes_{\Z} \Q \cong H^*(X;\Q)\otimes_{\Z} MU^*$. 
This implies that the sum $\oplus_{p\in \Z} J_{MU}^{2p-1}(X) \otimes \Q$ is a flat $MU^*$-module. 
Thus, for $\gamma \in \Mh^p(X)$, if $AJ(\gamma)$  is non-zero in $J_{MU}^{2p-1}(X) \otimes \Q$, then $AJ(\gamma) \cdot [\Pro^1]$ is non-zero in $J_{MU}^{2p-3}(X)\otimes \Q$ and therefore non-zero in $J_{MU}^{2p-3}(X)$. 
Now we can take an element $\gamma \in \Mh^p(X)$ such that $\varphi(\gamma) = 0$ and the image of $\gamma$ in $J^{2p-1}(X)$ is non-torsion. 
Then the above argument shows that $AJ(\gamma) \cdot [\Pro^1]$ is non-zero in $J_{MU}^{2p-3}(X)$. 
However, the image $\tau_J(AJ(\gamma) \cdot [\Pro^1])$ vanishes in $J^{2p-3}(X)$ since $\tau_0$ sends $[\Pro^1]$ to zero. 
Examples of this situation where $X$ is a projective smooth complex algebraic variety are described in \cite[Examples 7.15 and 7.16]{hfc}. \\

Finally, we look at $\coker \tau$. 
The most interesting case is that of a non-torsion element in $\coker \tau$ which induces an element in $\ker \tau_J$ that remains non-trivial after taking the tensor product with $\R/\Z$ over $MU^*$. 
For certain complex Lie groups, for example $SO(5)$, we can show that there are such elements in $\coker \tau$.  
However, we are so far not able to produce such elements for $X$ being compact or even projective.


\bibliographystyle{amsalpha}

\end{document}